\documentclass{amsart}
\usepackage{amssymb,amsmath,amsfonts,pictex,graphicx,fullpage,etex,tikz,hyperref}
\usetikzlibrary{shapes,shapes.arrows,decorations,positioning}

\newcommand{\grouppair}[2]{\left<#1,#2\right>}
\newcommand{\smallfrac}[2]{\mbox{\small $\frac{#1}{#2}$}}

\newcommand{\eS}{\mathcal{S}}
\newcommand{\Z}{\mathbb{Z}}

\newcommand{\F}{\mathbb{F}}
\newcommand{\R}{\mathbb{R}}
\newcommand{\E}{\mathbb{E}}

\newcommand{\eL}{\mathcal{L}}
\newcommand{\eLcat}{\mathcal{L}^{\textrm{CAT(0)}}}
\newcommand{\incl}{\hookrightarrow}

\newcommand{\Vtrans}{\overrightarrow{\sf t}}
\newcommand{\Hdisp}{{\sf D}}
\newcommand{\slopevec}{\overrightarrow{{\sf m}}}
\newcommand{\slopescal}{{\sf M}}
\newcommand{\unitvec}{\overrightarrow{{\sf u}}}
\newcommand{\nutoo}{{\sf N}}

\DeclareMathOperator{\Min}{Min}

\DeclareMathOperator{\id}{id}

\DeclareMathOperator{\limset}{limset}

\theoremstyle{plain}
\newtheorem{prop}{Proposition}[section]

\newtheorem{lemma}[prop]{Lemma}
\newtheorem{corollary}[prop]{Corollary}

\newtheorem{Thm}{Theorem}
\newtheorem*{Theorem}{Theorem}
\newtheorem{Co}[Thm]{Corollary}

\theoremstyle{definition}
\newtheorem{example}[prop]{Example}

\newtheorem{rem}[prop]{Remark}

\begin{document}

\title[Cell-Like Equivalences and Boundaries of CAT(0) Groups]{Cell-Like Equivalences and Boundaries of CAT(0) Groups}

\begin{abstract}
In 2000, Croke and Kleiner showed that a CAT(0) group $G$ can admit more
than one boundary. This contrasted with the situation for $\delta $%
-hyperbolic groups, where it was well-known that each such group admitted a
unique boundary---in a very stong sense. Prior to Croke and Kleiner's
discovery, it had been observed by Geoghegan and Bestvina that a weaker sort
of uniquness does hold for boundaries of torsion free CAT(0) groups; in
particular, any two such boundaries always have the same shape. Hence, the
boundary really does carry significant information about the group itself.
In an attempt to strengthen the correspondence between group and boundary,
Bestvina asked whether boundaries of CAT(0) groups are unique up to
cell-like equivalence. For the types of space that arise as boundaries of
CAT(0) groups, this is a notion that is weaker than topological equivalence
and stronger than shape equivalence.

In this paper we explore the Bestvina Cell-like Equivalence Question. We
describe a straightforward strategy with the potential for providing a fully
general positive answer. We apply that strategy to a number of test cases
and show that it succeeds---often in unexpectedly interesting ways.
\end{abstract}

\author{Craig Guilbault\and Christopher Mooney$^\ast$}
\date{\today}
\thanks{$^\ast$The second author was supported in part by NSF grant EMSW21-RTG: Training the Research Workforce in Geometry, Topology and Dynamics}
\maketitle

\section{Introduction}

One striking difference between the category of negatively curved groups and
that of nonpositively curved groups occurs at their ends; whereas a
$\delta$-hyperbolic group admits a topologically unique boundary, a CAT(0) group can
admit uncountably many distinct boundaries
\cite{croke-kleiner-nonrigid,wilson,mooney-knot_groups,mooney-generalizing_CK}.
On its surface, that observation might lead one to believe that a boundary
for a CAT(0) group is not a useful object, but that is not the case. Many
properties remain constant across the spectrum of boundaries of a given
CAT(0) group, and thus may be viewed as properties of the group itself. One
substantial such property, which implies many others, is the
\emph{shape} of the boundary. That observation was made indirectly by Geoghegan
\cite{geoghegan} and, specifically for CAT(0) groups, by
Bestvina \cite{bestvina-local_homology}. The upshot is that all boundaries of a given
CAT(0) group are topologically similar in a manner made precise by shape
theory---a classical branch of geometric topology developed specifically for
dealing with spaces with the sort of bad local properties that frequently
occur in boundaries of groups. Looking for an even stronger correlation
between CAT(0) groups and their boundaries, Bestvina posed the
following:\medskip 

\noindent \textbf{Bestvina's Cell-like Equivalence Question. }\emph{For a
given CAT(0) group }$G$,\emph{\ are all boundaries cell-like equivalent?}%
\medskip

Precise formulations of the notion of `shape equivalence' and `cell-like
equivalence' and their relationship to one another will be given shortly.
For now we give a quick description of the concept of cell-like equivalence
to aid in painting the big picture.

A pair of compacta $X$ and $Y$ are declared to be cell-like equivalent if
there exists a third compactum $Z$ and a pair of cell-like maps $X\overset{%
f_{1}}{\longleftarrow }Z\overset{f_{2}}{\longrightarrow }Y$. (The reader may
temporarily think of a cell-like map as a surjective map with contractible point
preimages.) To obtain an equivalence relation we permit several intermediate
spaces: $X$ and $Y$ are declared to be cell-like equivalent if there exists
a diagram of compacta and cell-like maps of the form:%
\begin{equation}
\begin{tabular}{ccccccccc}
& $Z_{1}$ &  & $Z_{3}$ &  &  &  & $Z_{2n+1}$ &  \\ 
& $\swarrow \qquad \searrow $ &  & $\swarrow \qquad \searrow $ &  & $%
\swarrow \quad \cdots \quad \searrow $ &  & $\swarrow \qquad \searrow $ & 
\\ 
$X$ &  & $Z_{2}$ &  & $Z_{4}$ &  & $Z_{2n}$ &  & $Y$%
\end{tabular}
\label{diagram:zig-zag}
\end{equation}%
Clearly, cell-like equivalence is weaker than topological equivalence;
moreover, if we require that all spaces involved be finite-dimensional,
cell-like equivalence is stronger than shape equivalence \cite{sher}.
In addition to
lying between the notions of topological equivalence and shape equivalence,
cell-like equivalence has the advantage of allowing an easily understood
equivariant variation. Compacta $X$ and $Y$, each equipped with a $G$%
-action, are declared to be `$G$-equivariantly cell-like equivalent' if
there exists a diagram of type (\ref{diagram:zig-zag}) for which each of
the $Z_{i}$ also admits a $G$-action, and each of the cell-like maps respects
the corresponding actions. Bestvina has indicated an interest in the
following:\medskip

\noindent \textbf{Bestvina's Equivariant Cell-like Equivalence Question. }%
\emph{For a given CAT(0) group }$G$,\emph{\ are all boundaries }$G$\emph{%
-equivariantly cell-like equivalent?}\medskip

In this paper we propose a general strategy for obtaining an affimative
solution to the equivariant version of Bestvina's question. That strategy is
straighforward; it is described at the end of this section. Thus far we are
unable to complete the program for arbitrary CAT(0) groups. Instead we
present some specific cases where our strategy works---sometimes in
surprising ways. In addition we develop some potentially useful
generalizations of our approach.

In the remainder of this introduction we review some basic notions related
to shape theory and cell-like equivalences. We then discuss CAT(0) spaces
and groups enough so we can go on to describe our standard strategy for
obtaining equivariant cell-like equivalences between pairs of boundaries.
Lastly we outline the main examples and results to be presented in the
remainder of the paper.

By \emph{compactum} we mean a compact metric space. There are a variety of
ways of saying what it means for compacta $X$ and $Y$ to be shape equivalent
(denoted $X\overset{\text{sh}}{\sim }Y$). One method,
due to Chapman \cite{chapman-hilbert}, involves the Hilbert
cube $Q=\prod_{i=1}^{\infty }\left[ 0,1\right] $. Embed $X$ and $Y$ in as $%
\mathcal{Z}$-sets (for example, place $X$ and $Y$ in $\left\{ 0\right\}
\times \prod_{i=2}^{\infty }\left[ 0,1\right] \subseteq Q$). Then $X\overset{%
\text{sh}}{\sim }Y$ if and only if $Q-X$ is homeomorphic to $Q-Y$. If $X$
and $Y$ are both finite-dimensional one can avoid infinite dimensional
topology by embedding $X$ and $Y$ nicely in $\mathbb{R}^{n}$, where $n$ is
large compared to the dimensions of $X$ and $Y$. Then $X\overset{\text{sh}}{%
\sim }Y$ if and only if $\mathbb{R}^{n}-X$ is homeomorphic to $\mathbb{R}%
^{n}-Y$. See \cite{sher} for details.

Another way to characterize shape is more complex, but often easier to
apply. Given a compactum $X$, we first choose an \emph{associated inverse
sequence} of finite polyhedra and continuous maps%
\begin{equation}
K_{0}\overset{f_{1}}{\longleftarrow }K_{1}\overset{f_{2}}{\longleftarrow }%
K_{2}\overset{f_{2}}{\longleftarrow }\cdots \text{.}
\label{inverse sequence}
\end{equation}%
If $X$ happens to arise as an inverse limit of finite polyhedra, then that
sequence may be chosen as the associated inverse sequence. Another way to
obtain an associated inverse sequence for $X$ is to choose a sequence of
finite covers $\left\{ \mathcal{U}_{i}\right\} _{i=0}^{\infty }$ of $X$ by $%
\varepsilon _{i}$-balls such that $\varepsilon _{i}\rightarrow 0$ and each $%
\mathcal{U}_{i}$ refines $\mathcal{U}_{i-1}$; then, for each $i$, let $K_{i}$
be the nerve of $\mathcal{U}_{i}$ and $f_{i}$ be the corresponding
simplicial map. Yet another way of obtaining an associated inverse sequence
can be applied when $X$ is finite-dimensional: embed $X$ in $\mathbb{R}^{n}$
and let $\left\{ K_{i}\right\} _{i=0}^{\infty }$ be a decreasing sequence of
polyhedral neighborhoods of $X$ with all bonding maps being inclusions.

Given an increasing sequence $\left\{ i_{k}\right\} _{k=0}^{\infty }$ of
natural numbers, there is a corresponding subsequence of (\ref{inverse
sequence})%
\begin{equation*}
K_{i_{0}}\overset{f_{i_{1},i_{0}}}{\longleftarrow }K_{i_{1}}\overset{%
f_{i_{2},i_{1}}}{\longleftarrow }K_{i_{2}}\overset{f_{i_{3},i_{2}}}{%
\longleftarrow }\cdots 
\end{equation*}%
where, for any integers $n>m$, $f_{n,m}$ is the obvious composition of the $%
f_{j}$ taking $K_{n}$ to $K_{m}$. Declare a pair of inverse sequences $%
\left\{ K_{i},f_{i}\right\} $ and $\left\{ L_{i},g_{i}\right\} $ to be \emph{%
pro-equivalent} if they contain subsequences that fit into a diagram of the
form%
\begin{equation}
\begin{array}{ccccccc}
K_{i_{0}} & \overset{f_{i_{1},i_{0}}}{\longleftarrow } & K_{i_{1}} & \overset%
{f_{i_{2},i_{1}}}{\longleftarrow } & K_{i_{2}} & \overset{f_{i_{3},i_{2}}}{%
\longleftarrow } & \cdots  \\ 
& \nwarrow \quad \swarrow  &  & \nwarrow \quad \swarrow  &  & \nwarrow \quad
\swarrow  &  \\ 
& L_{j_{0}} & \overset{g_{j_{1},j_{0}}}{\longleftarrow } & L_{j_{1}} & 
\overset{g_{i_{2},i_{1}}}{\longleftarrow } & L_{j_{2}} & \cdots 
\end{array}
\label{diagram:ladder}
\end{equation}%
where each triangle is only required to commute up to homotopy. Now $X%
\overset{\text{sh}}{\sim }Y$ if and only if associated inverse sequences are
pro-equivalent.

\begin{example}
\label{Ex: shape circles}By using the inverse sequence approach, it is easy
to see that the following examples each has the shape of a circle.
\begin{figure}[!ht]
	\begin{minipage}[b]{0.45\linewidth}
		\centering
		\includegraphics[width=0.7\linewidth]{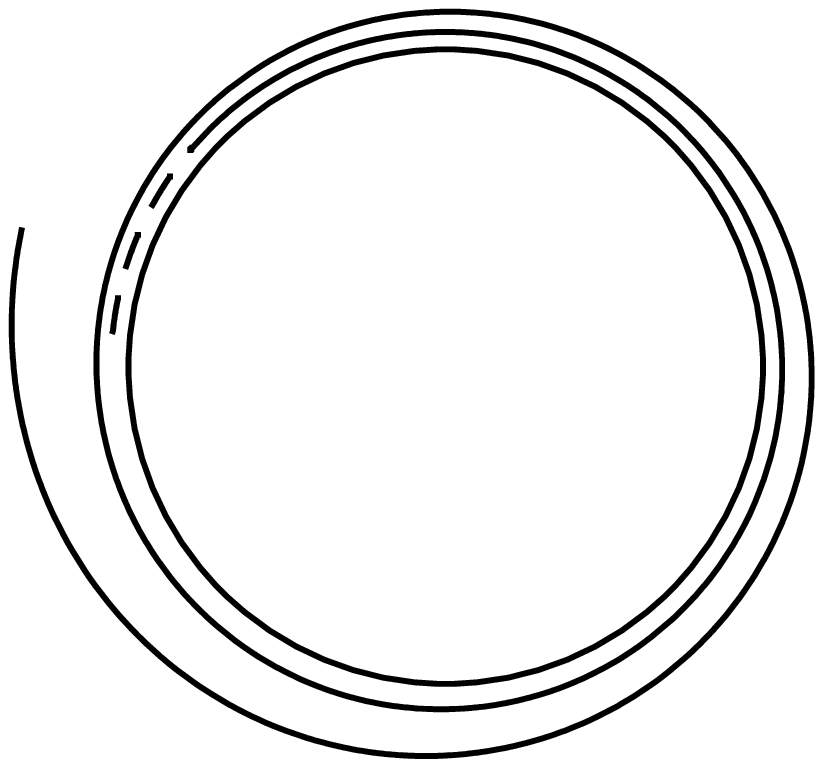}
		\caption{Ferry Spiral}
		\label{fig:ferry spiral}
	\end{minipage}
	\begin{minipage}[b]{0.45\linewidth}
		\centering
		\includegraphics[width=0.7\linewidth]{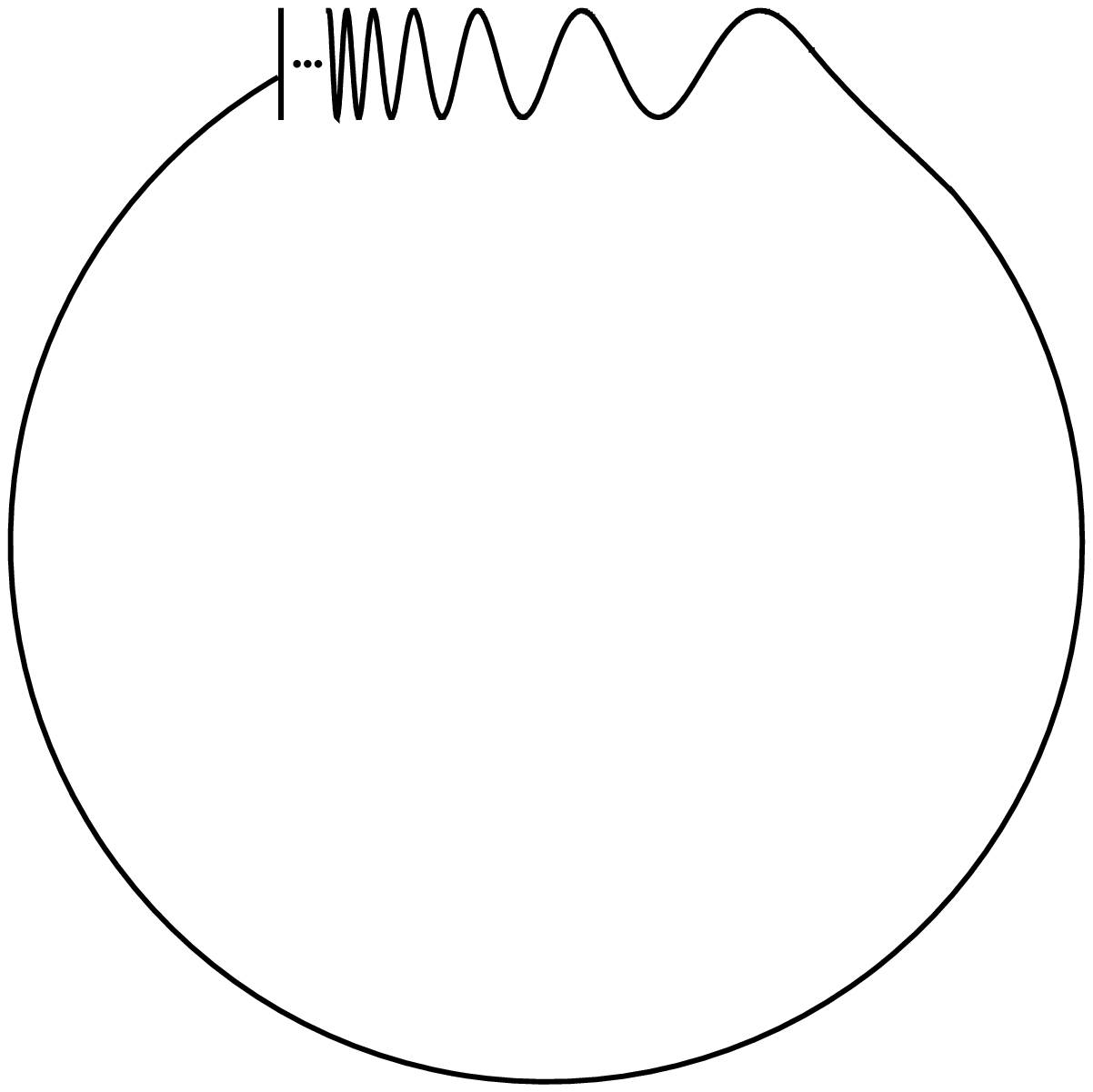}
		\caption{Warsaw Circle}
		\label{fig:warsaw circle}
	\end{minipage}
\end{figure}
\end{example}

\begin{example}
\label{Ex: sin(1/x)}By repeated application of Borsuk's homotopy extension
property, one sees that every contractible compactum has the same shape as a
point; these are the prototypical comapcta with trivial shape. An example of
a non-contractible compactum with the shape of a point is the topologist's
sine curve:
\begin{figure}[ht!]
	\centering
	\includegraphics[width=0.3\linewidth]{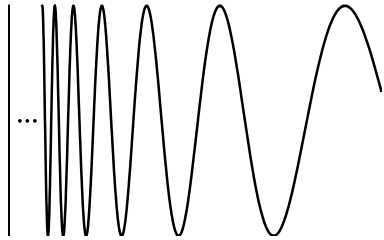}
	\caption{Topologist's Sine Curve}
	\label{fig:sine}
\end{figure}
\end{example}

\subsection{The notion of cell-like equivalence}

A compactum is \emph{cell-like }if it has the shape of a point. A map
$f:X\rightarrow Y$ between compacta is cell-like if $f^{-1}\left( y\right) $
is cell-like for each $y\in Y$. Cell-like maps have been studied
extensively---they play a central role in manifold topology. Compacta $X$
and $Y$ are \emph{cell-like equivalent} if there exists a finite sequence of
compacta $Z_{1},\cdots ,Z_{2n+1}$ $(n\geq 0)$ and cell-like maps as
described by diagram (\ref{diagram:zig-zag}); in this case we write $X%
\overset{\text{CE}}{\sim }Y$. When $X$ and $Y$ are finite-dimensional, the
existence of a cell-like map $f:X\rightarrow Y$ implies $X\overset{\text{sh}}%
{\sim }Y$; thus, if $X$ and $Y$ are `cell-like equivalent through
finite-dimensional compacta' (there exists a diagram of type (\ref{diagram:zig-zag})
for which all spaces are finite-dimensional), then $X$ and $Y$
have the same shape. Since all boundaries of CAT(0) groups are
finite-dimensional \cite{swenson}, as are all intermediate spaces utilized in
this paper, we generally think of `cell-like equivalence' as being stronger
than `shape equivalence'.

A famous counterintuitive example helps to illustrate the above definition.

\begin{example}
Let $A=\left[ 0,1\right] \times \left\{ 0\right\} \subseteq \mathbb{R}^{2}$
and $C\subseteq A$ be the middle-thirds Cantor set; let $B\subseteq \mathbb{R%
}^{2}$ be the cone over $C$ with cone-point $\left( \frac{1}{2},1\right) $.
Let $Z=A\cup B$.
Since $A$ and $B$ are contractible the quotient maps $Z\to Z/A$ and
$Z\to Z/B$ are cell-like.  The image of the latter
is the standard Hawaiian earring with countably many loops; call this space $X$.
Notice that $Z/A$ is homeomorphic to $\Sigma C$, the suspension
of a Cantor set. By crushing out a single suspension arc, we get a cell-like
map $\Sigma C\rightarrow Y$ where $Y$ is a `bigger Hawaiian earring', having
a Cantor set's worth of loops. Putting all of this together, we get that
$X$ and $Y$ are cell-like and shape equivalent (see Figure \ref{fig:hawaiian}).

\begin{figure}[ht!]
\begin{tabular}{ccccccc}
&& \includegraphics[width=0.17\linewidth]{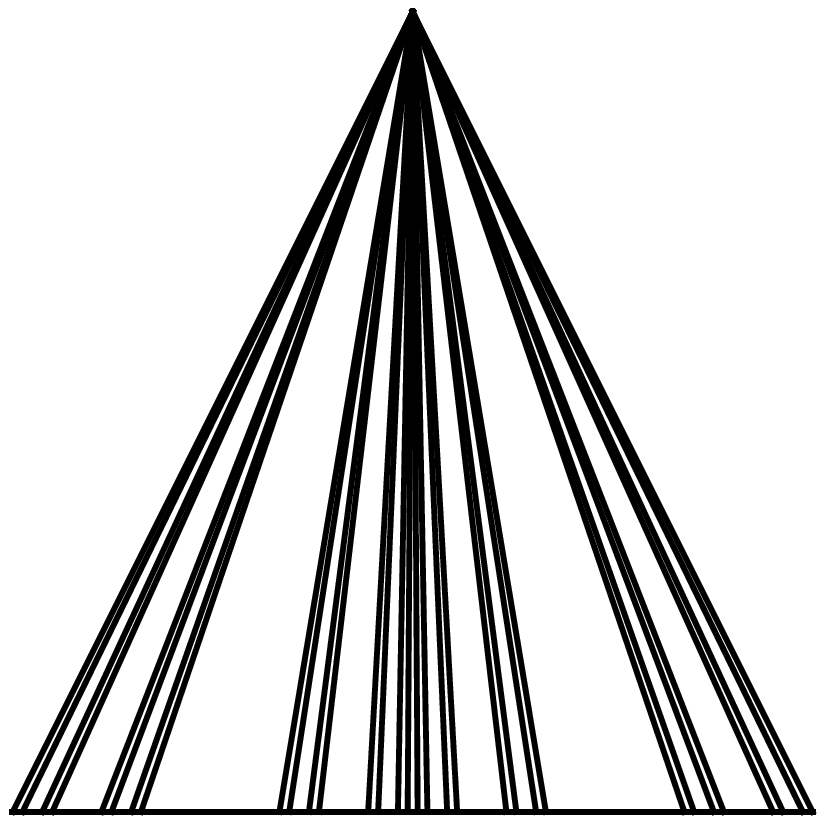}\\
&& $Z$\\
&\includegraphics[width=0.045\linewidth]{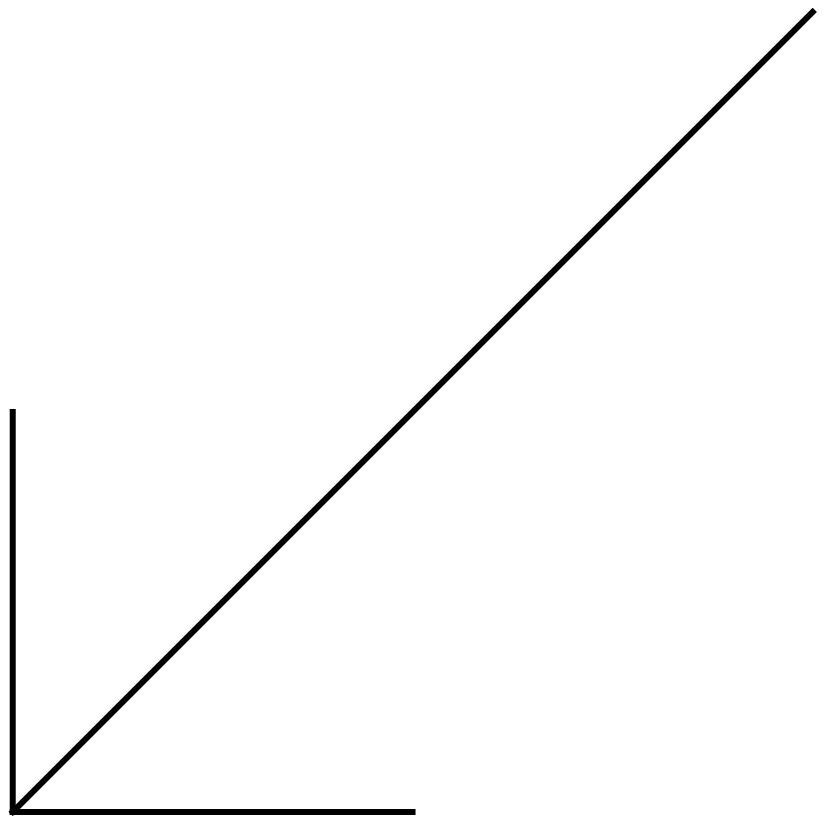}&&\includegraphics[width=0.05\linewidth]{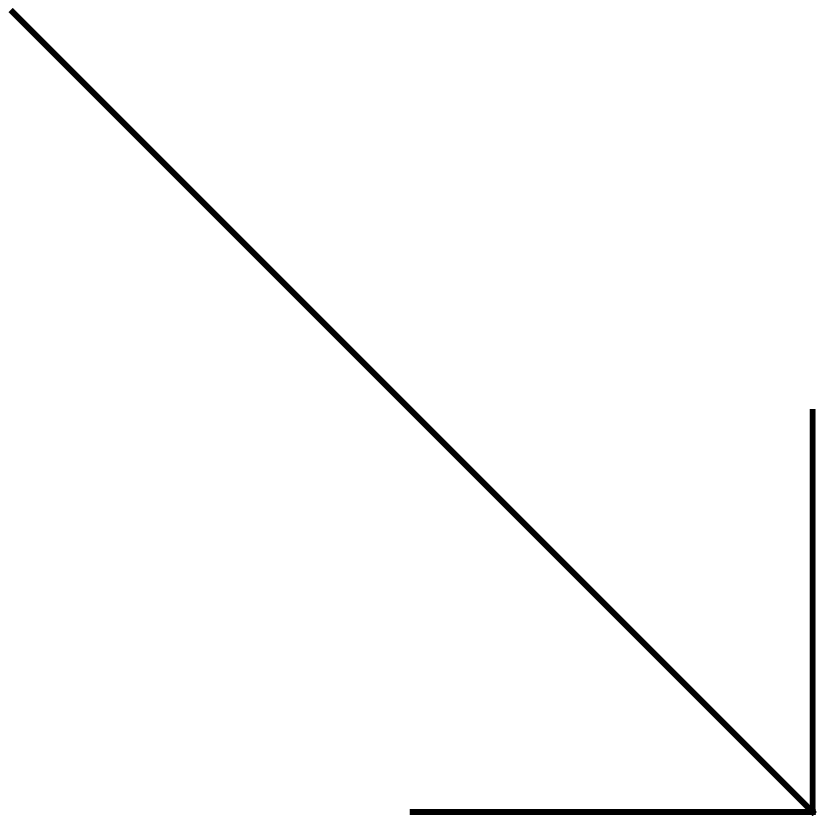}\\
\includegraphics[width=0.17\linewidth]{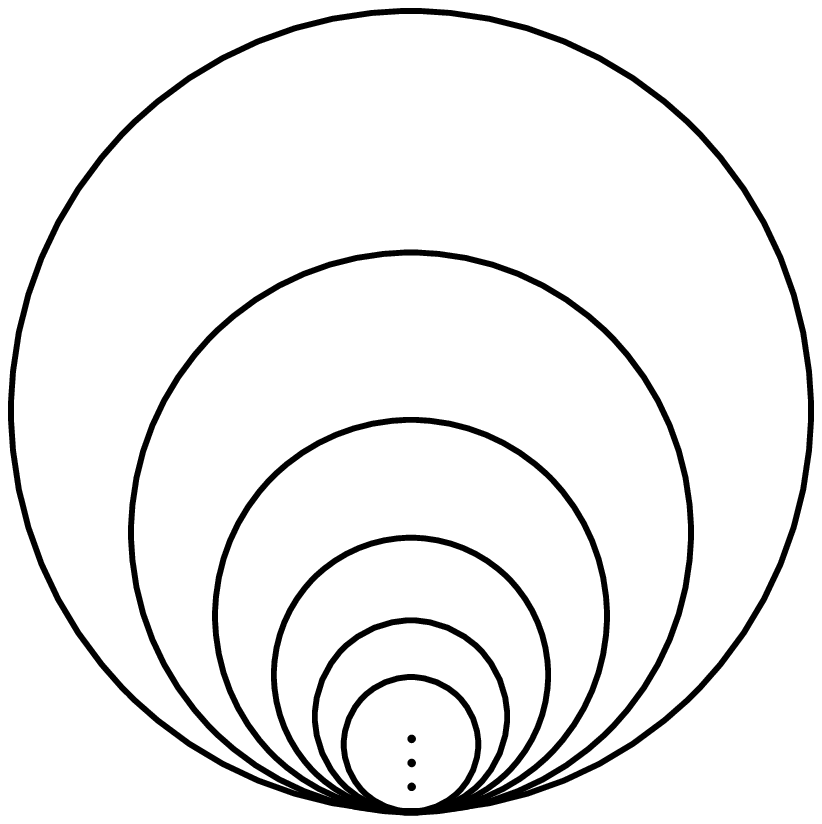} &&&&
\includegraphics[width=0.17\linewidth]{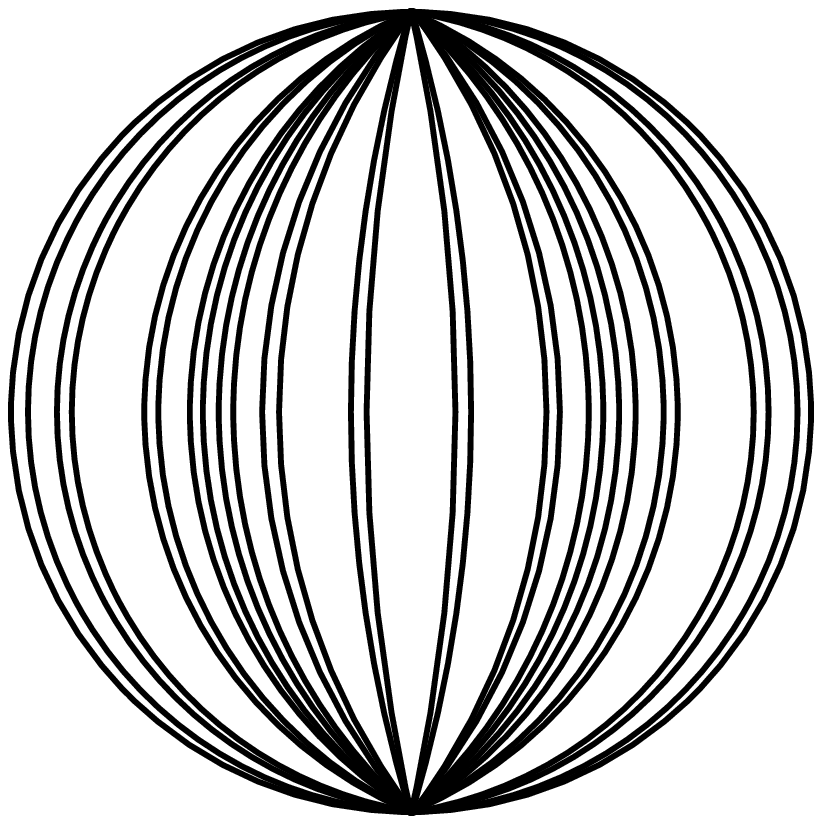} &
\parbox[h]{0.05\linewidth}
{
	\vspace{-2cm}
	\includegraphics[width=1\linewidth]{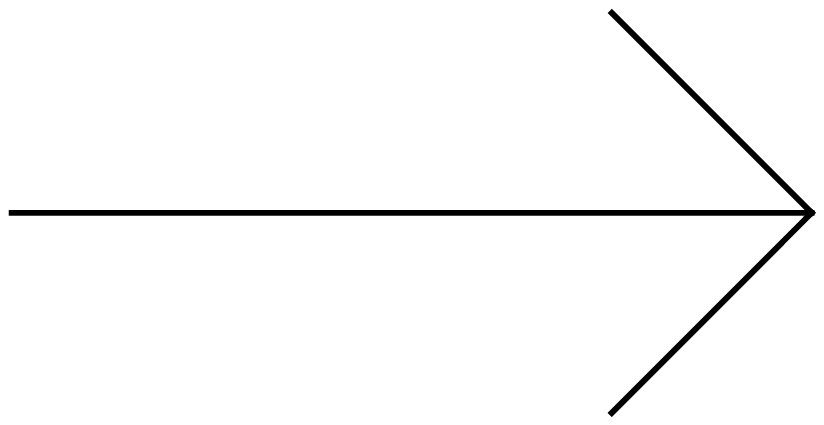}
}
&
\includegraphics[width=0.17\linewidth]{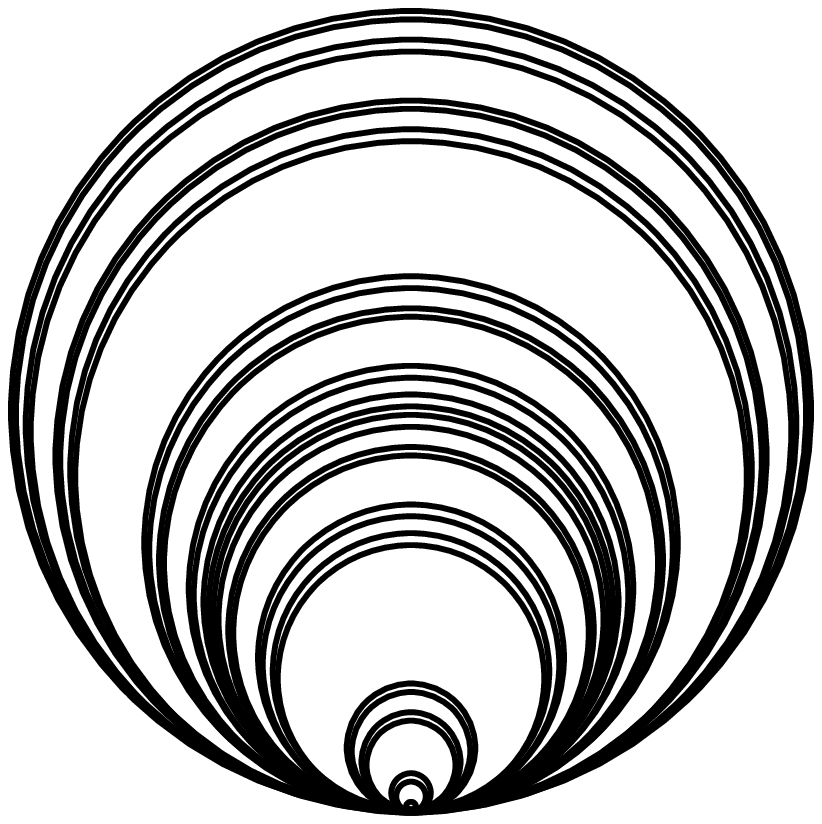} \\
$X=Z/B$ &&&& $Z/A$ && $Y$ \\
\end{tabular}
\caption{CE Hawaiian Earrings}
\label{fig:hawaiian}
\end{figure}
\end{example}

\begin{example}
An easy application of Example \ref{Ex: sin(1/x)} shows that the Warsaw
circle (see Figure \ref{fig:warsaw circle})
is cell-like equivalent to an ordinary circle.
In \cite{ferry-CE_not_SH}, it is shown that the Ferry Spiral
(see Figure \ref{fig:ferry spiral}) is not cell-like equivalent to a circle.
\end{example}

As noted in the introduction, the notion of cell-like equivalence lends
itself nicely to an equivariant version. Suppose compacta $X$ and $Y$ each
admit a $G$-action and suppose there exists a diagram of the form
(\ref{diagram:zig-zag}) where each of the compacta $Z_{i}$ also admits a $G$%
-action and each of the maps in that diagram commutes with the appropriate
actions. We say that $X$ and $Y$ are $G$\emph{-equivariantly cell-like
equivalent} and we write $X\overset{G\text{-CE}}{\sim }Y$.

\subsection{CAT(0) groups and their boundaries}

A geodesic metric space $X$ is called a \textit{CAT(0) space} if each of its triangles is at
least as thin as the corresponding comparison triangle in the Euclidean
plane. A group $G$ is called a \textit{CAT(0) group} if it acts
\textit{geometrically} (properly and
cocompactly via isometries) on a proper CAT(0) space. Such a space $X$ can
be compactified by the addition of its \textit{visual boundary} $\partial X$ which
may be defined as the space of all equivalence classes of geodesic rays in $%
X $, where a pair of rays $\alpha ,\beta :[0,\infty )\rightarrow X$ are
equivalent if they are asymptotic, i.e., if $\left\{ d\left( \alpha \left(
t\right) ,\beta \left( t\right) \mid t\in \lbrack 0,\infty )\right) \right\} 
$ is bounded above. When $G$ acts geometrically on $X$ we call $\partial X$
a boundary for $G$. Clearly, the action of $G$ on $X$ induces an action by $%
G $ on $\partial X$. An alternative, but equivalent, approach to
defining the visual boundary declares $\partial X$ to be the collection of
all geodesic rays emanating from a fixed base point $x_{0}\in X$. This
simplifies matters since no equivalence relation is needed; we will use the latter
approach when convenient.
We put the \textit{cone topology} on $\partial X$ which roughly says that two geodesic rays are close
in $\partial X$ if they track together for a long time before they diverge.
More details on CAT(0) spaces and their
boundaries, including a discussion of the topology on $\overline{X}=X\cup
\partial X$, will be presented as necessary. In addition, the reader may
wish to consult \cite{bridson-haefliger}.

Nonuniqueness of the boundary of a CAT(0) group $G$ is possible since $G$
can act on more than one CAT(0) space. When the action by $G$ is free,
covering space techniques and other topological tools allowed Bestvina
\cite{bestvina-local_homology}
to show that all boundaries of $G$ are shape equivalent. Later,
Ontaneda \cite{ontaneda} extended that obsevation to include \emph{all} CAT(0)
groups. In those cases where all CAT(0) boundaries of a given $G$ are
homeomorphic we say that $G$ is \emph{rigid}.
Clearly Bestvina's Cell-like Equivalence Question has a positive answer for all such groups.
A positive answer has also been given for groups which split as products
with infinite factors \cite{mooney-CE}.

When a group $G$ acts nicely on multiple spaces, a key relationship between
those spaces is captured by the notion of `quasi-isometry'. A function $%
f:X\rightarrow X^{\prime }$ between metric spaces is called a \emph{%
quasi-isometric embedding (QIE)} if there exist positive constants $\lambda $ and $%
\varepsilon $ such that for all $x,y\in X$%
\begin{equation*}
\frac{1}{\lambda }d\left( x,y\right) -\varepsilon \leq d^{\prime }\left(
f\left( x\right) ,f\left( y\right) \right) \leq \lambda d\left( x,y\right)
+\varepsilon \text{.}
\end{equation*}%
If, in addition, there exists a $C>0$ such that for every $z\in X^{\prime }$%
, there is some $x\in X$ such that $d^{\prime }\left( f\left( x\right)
,z\right) \leq C$, then we call $f$ a \emph{quasi-isometry} and declare $X$
and $X^{\prime }$ to be \emph{quasi-isometric}.

By choosing a finite generating set and endowing it with the corresponding
word metric, any finitely generated group can be viewed as a metric space.
It follows from the \v{S}varc-Milnor Lemma that, up to
quasi-isometry, this metric space is independent of the choice of generating
set; in fact if $X$ is any length space on which $G$ acts geometrically,
then for any base point $x_{0}\in X$ the orbit map $G\to X$ given by $g\mapsto gx_{0}$ is a quasi-isometry
\cite{svarc,milnor-note}.

Given a subset $A$ of a CAT(0) space $X$, define the \emph{limset} of $A$ to
be the collection of all limit points of $A$ lying in $\partial X$, in
other words, $\limset A=\overline{A}-X$ where the closure is taken
in $\overline{X}$. Clearly any such limset is a closed subset of $\partial X$.
If $G$ acts on a proper CAT(0) space properly discontinuously by isometries, then
we denote by $\limset(X,G)$ the limset of the image of $G$ under the orbit
map.  This provides a compactification $G\cup\limset(X,G)$ for $G$.
Note that if this action is geometric then $\limset(X,G)=\partial X$.
If $G$ acts properly discontinuously on two proper
CAT(0) spaces $X$ and $Y$, then it is natural to compare the two compactifications
$\partial X$ and $\partial Y$.  If the identity map on $G$ extends
continuously to a map $G\cup\partial X\to G\cup\partial Y$, then the restriction
$\partial X\to\partial Y$ is called a \textit{limset map}.
The existence of such a map is very strong.  It means that whenever an unbounded
sequence of group elements converges in one compactification, it also
converges in the other (see Lemma \ref{le:characterization of existence of limset maps}).
Two limsets are considered \textit{equivalent} if there is a limset map between them
which is a homeomorphism.

We call $G$ \emph{strongly rigid} if whenever $G$ acts geometrically on
proper CAT(0) spaces $X$ and $Y$, the boundaries $\partial X$ an $\partial Y$
are equivalent in the above sense.
Examples of such groups include free abelian groups, $\delta $-hyperbolic
CAT(0) groups (hereafter referred to as \emph{negatively curved} groups),
and others \cite{kleiner-leeb,hruska-kleiner}.
Clearly Bestvina's Equivariant Cell-like Equvalence Question has a positive
answer for all strongly rigid groups.

If $G$ acts properly discontinuously by isometries on $X$, then so does any subgroup
$H\leq G$.  If $H$ has infinite index, then it does not act geometrically,
since cocompactness has been lost. Moreover, even when $H$ is finitely
generated, it is not always the case that $H\hookrightarrow G$ (or
equivalently $h\mapsto hx_{0}$) is a QIE.
An object of special interest to us will be 
$\limset H$ for certain subgroups $H$ of CAT(0) groups.

Beyond the above mentioned examples of strongly rigid groups,
the question of when two boundaries of a CAT(0) group are equivalent has
been studied by Croke and Kleiner for a class of groups including graph manifold
groups \cite{croke-kleiner-graph_manifolds}, and by Hosaka in more generality
\cite{hosaka-strong_rigidity}.

\subsection{The standard strategy and our Main Conjecture}
Suppose $G$ acts geometrically on a pair of proper CAT(0) spaces $X_{1}$ and 
$X_{2}$. Then the $l_2$-metric $d=\sqrt{d_{1}^{2}+d_{2}^{2}}$ makes $%
X_{1}\times X_{2}$ a proper CAT(0) space on which $G\times G$ acts
geometrically via the product action. It is a standard fact that $\partial
\left( X_{1}\times X_{2}\right) $ is homeomorphic to the topological
join of the original boundaries \cite[Example II.8.11(6)]{bridson-haefliger}. To see this, first choose a base point $%
\left( x_{1},x_{2}\right) \in X_{1}\times X_{2}$ and define slopes of
segments and rays in $X_{1}\times X_{2}$ based at $\left( x_{1},x_{2}\right) 
$ in the obvious way. A ray $\alpha $ may be projected into $X_{1}$ and $%
X_{2}$ to obtain a pair of rays $\alpha _{1}$ and $\alpha _{2}$ ---except in
those cases where the slope is $0$ or $\infty $ which produce an $\alpha _{i}
$ that is constant. Assign to each $\alpha$ three coordinates: $\alpha_{1}$,
$\alpha_{2}$, and the slope of $\alpha$.  Keeping in mind the
exceptional cases where $\alpha $ has slope $0$ or $\infty $, we get a
correspondence between $\partial \left( X_{1}\times X_{2}\right) $ and the
quotient space%
\begin{equation*}
\partial X_{1}\ast \partial X_{2}=\partial X_{1}\times \partial X_{2}\times 
\left[ 0,\infty \right] /\sim 
\end{equation*}%
where $\left( \alpha _{1},\alpha _{2},0\right) \sim \left( \alpha
_{1},\alpha _{2}^{\prime },0\right) $ for all $\alpha _{2},\alpha
_{2}^{\prime }\in \partial X_{2}$ and $\left( \alpha _{1},\alpha _{2},\infty
\right) \sim \left( \alpha _{1}^{\prime },\alpha _{2},\infty \right) $ for
all $\alpha _{1},\alpha _{1}^{\prime }\in \partial X_{2}$. This join
contains a \emph{preferred copy} of $\partial X_{1}$ (all points with slope $%
0$) and a \emph{preferred copy} of $\partial X_{2}$ (all points with slope $%
\infty $) which may be identified with the boundaries of convex subspaces $%
X_{1}\times \left\{ x_{2}\right\} $ and $\left\{ x_{1}\right\} \times X_{2}$%
. A more thorough development of the the notion of slope may be found in
Section \ref{sec:schmear}.

Now consider the diagonal subgroup $G^{\Delta }=\{\left( g,g\right) \mid
g\in G\}$ of $G\times G$. Clearly, $G^{\Delta }$ is isomorphic to $G$ and
acts on $X_{1}\times X_{2}$ properly by isometries.
For $g\in G$, we will denote $g^\Delta=(g,g)$.
In Section \ref{subsection:intermediatelimset}, we make the following observations:\medskip

\begin{itemize}
\item[(i)] The map $g\longmapsto g^\Delta \left( x_{1},x_{2}\right) 
$ is a QIE of $G$ in $X_{1}\times X_{2}$, and

\item[(ii)] $\limset G^{\Delta }$ is a closed subset of $\partial
X_{1}\ast \partial X_{2}$ that misses the preferred copies of $\partial
X_{1} $ and $\partial X_{2}$.\medskip
\end{itemize}

\noindent We refer to $\limset G^{\Delta }$ as a \emph{schmear} of 
$\partial X_{1}$ and $\partial X_{2}$. Item (i) is used in proving (ii) and
offers hope that $\Lambda $ resembles a \textit{boundary} for $G$. Item (ii) allows
us to restrict the projections of $\partial X_{1}\times \partial X_{2}\times
(0,\infty )$ onto $\partial X_{1}$ and $\partial X_{2}$ to obtain a pair of $%
G$-equivariant \emph{schmear maps} $\phi _{1}:\Lambda \rightarrow \partial
X_{1}$ and $\phi _{2}:\Lambda \rightarrow \partial X_{2}$.

Our standard strategy is summed up by the following:\medskip

\noindent \textbf{Main Conjecture. }\emph{Suppose }$G$\emph{\ acts
geometrically on a pair of CAT(0) spaces }$X_{1}$\emph{\ and }$X_{2}$\emph{.
Then both schmear maps are cell-like; hence $\partial X_{1}$\emph{\ and }$%
\partial X_{2}$ are }$G$\emph{-equivariantly cell-like equivalent. }\medskip 

\subsection{The main results}

As noted earlier, we are not yet able to make the above program work in
full generality. In this paper, we provide positive evidence for our
approach by presenting an array of interesting examples and proving the
conjecture for a specific class of groups.

One simple, but revealing, example involves the negatively curved
group $\mathbb{F}_{2}$ which has boundary homeomorphic to a Cantor set $C$.
By strong rigidity for $\delta $-hyperbolic groups even the Equivariant
Bestvina Question is not in doubt here; one might even expect any schmear
for $\mathbb{F}_{2}$ to be just another copy of $C$. On the contrary, by
choosing different CAT(0) spaces on which $\mathbb{F}_{2}$ acts, the
resulting schmear is often not a Cantor set, but rather, is homeomorphic to $%
C\times \left[ 0,1\right] $. In those cases, all point preimages are copies
of $\left[ 0,1\right] $. See Example \ref{ex:schmear}.
Indeed we can compute schmears of CAT(0) boundaries of negatively curved groups
by applying recent work of Link \cite{link} (see Corollary \ref{Co:Link}).

In a rather different way the group $\mathbb{F}_{2}\times \mathbb{Z}$ admits
an interesting schmear.  Here we can construct a nontrivial schmear for a pair
of boundaries where the two boundaries are the boundary of the same space,
showing that the schmear depends not only on the underlying spaces themselves, but
also on the action chosen.
In \cite{bowers-ruane} Bowers and Ruane utilized a
standard and a twisted action of $\mathbb{F}_{2}\times \mathbb{Z}$ on $%
T\times \mathbb{R}$, where $T$ is the standard valence four tree with edge
lengths equal to $1$, to show that this group is rigid but not strongly
rigid. Inserting those actions into our program produces a pair of schmear
maps for which some point preimages are intervals and the rest are
singletons. In some sense the schmear and its corresponding maps provide a
missing link between the standard and twisted action.
See Section \ref{examples of schmears}.

Our first theorem proves the existence of schmears.
Note that the statement applies to a broader class of group actions on CAT(0) spaces
than simply geometric actions.

\begin{Thm}
\label{Thm:Schmear}
Assume an infinite group $G$ acts on CAT(0) spaces $X_1$ and $X_2$ such that
$G\to X_1$ and $G\to X_2$ are QIEs.
Then there exists an action of $G$ by isometries
on a third CAT(0) space $X$ such that $G\to X$ is a QIE and there are natural limset
maps $\limset G\to\partial X_i$.
If the action of $G$ on both $X_i$ is by semi-simple isometries, then so is the action on $X$.
\end{Thm}

It follows from \cite{bowers-ruane} that Bestvina's Equivariant Cell-Like Equivalence Question
has a positive answer for products of negatively curved groups with free abelian groups.
They prove that any pair of boundaries is equivariantly homeomorphic.
Since this equivariant homeomorphism does not come from the orbit map,
this next theorem, which verifies the Main Conjecture for a particular subclass of such groups,
is stronger.

\begin{Thm}
\label{Thm:CE}
Assume $G=\F_m\times\Z^d$ acts geometrically on two CAT(0) spaces $X_1$ and $X_2$ and $\Lambda$ be the schmear
of the pair $\partial X_1$,$\partial X_2$.  Then point preimages for the schmear maps $\Lambda\to\partial X_i$ are
topological cells.
\end{Thm}

In other words, regardless of the CAT(0) spaces and geometric actions chosen, point
preimages under the schmear maps are homeomorphic to cells of various
dimensions. This is somewhat surprising when compared to recent work by
Staley \cite{staley} who uses these same groups to realize some exotic limsets
for images of geodesic rays under equivariant quasi-isometries. The lesson learned from
our theorem seems to be that \textquotedblleft taking the whole
schmear\textquotedblright\ has a tendency to paint over oddities in the
local behavior of limsets. Close examination of the examples and results
presented here will help make sense of this last comment.

Our final result addresses a natural question to this approach.  We have constructed
a single intermediate compactum $Z$ admitting an action by the group $G$ in question
by letting $Z$ be the limset of $G$ under some non-cocompact action.  One may ask
``Can $Z$ be realized as a boundary of $G$?''  The answer is ``No''.
If we seek to find equivariant cell-like equivalences between two boundaries of $G$
by taking the extensions of QIEs, then we must pass through limsets of non-cocompact
actions which cannot themselves be realized as boundaries of the group in question.

\begin{Thm}
\label{Thm:NotComparable}
Let $\F_2\times\Z$ act geometrically on two CAT(0) spaces $X$ and $Y$.  If there
is a limset map $\partial X\to\partial Y$, then that limset map is a homeomorphism.
\end{Thm}

\begin{rem}
If $H\le G$ is a finite index subgroup of a group $G$, then information about
$G$ gives information about $H$ -- if $G$ is CAT(0), then so is $H$ and every boundary
of $G$ is also a boundary of $H$.  The reverse direction is not understood.
It is currently an open question whether $H$ being CAT(0) implies that $G$ is also CAT(0),
and there are cases where $H$ and $G$ are both CAT(0), but $G$ has fewer boundaries than $H$.
One nice aspect of the approach taken in this paper is that all of our results which apply to
a given CAT(0) group H are immediately valid for any CAT(0) group containing H as a finite index subgroup
(see Proposition \ref{prop:finiteindex}).
\end{rem}

\subsection*{Acknowledgements}
We thank Kim Ruane.

\section{Examples}
\label{sec:examples}
In this section we provide concrete examples of the interesting end behavior of some simple CAT(0) groups as well as schmears.

\subsection{Boundaries of \texorpdfstring{$\F_2\times\Z$}{F2 x Z}}
It was first observed by Bowers and Ruane in \cite{bowers-ruane} that in contrast to the situation for negatively curved
and free abelian groups, equivariant quasi-isometries need not extend to homeomorphisms of boundaries.
This is true even
for $\F_2\times\Z$, which is perhaps the simplest example of a CAT(0) group which is neither negatively curved nor
free abelian.  In this section we provide several
variations on this example.

The main theorem of \cite{bowers-ruane} says that whenever \(\F_2\times\Z\) acts geometrically on a CAT(0) space $X$,
\(\partial X\) is homeomorphic to the suspension of the Cantor set.  By the Flat Torus Theorem, one can assume that
$X$ splits as a product $Y\times\R$ where $\Z$ is generated by a translation in the $\R$-coordinate
and the action of $\F_2$ projects to a geometric action on $Y$.  The suspension points of \(\partial X\)
are called \textit{poles}.  The subspace \(\partial Y\), which is homeomorphic to the Cantor set,
is called the \textit{equator}.  The suspension arcs are called \textit{longitudes}.

Typically longitudes are parameterized using angles.  For us it will be convenient to parameterize them in terms of slopes:
\[
	\partial X=\partial Y\times[-\infty,\infty]/\sim
\]
where $\sim$ collapses the sets \(Y\times\{\infty\}\) and \(Y\times\{-\infty\}\) to the poles.
Take a geodesic ray $\alpha$ based at a point \((y_0,0)\in Y\times\R\) which does not go to a pole.
Let $\alpha_Y$ be its projection onto
\(Y\times\{0\}\); \(\alpha_Y(\infty)\) is a point of the equator.
Then $\alpha$ lives in the half-plane \(F=\alpha_Y\times\R\) and has a slope $M(\alpha)$ defined in
terms of these coordinates.
The boundary of $F$ is the longitude containing $\alpha(\infty)$.  Given \(\zeta\in\partial X\), we will denote by
$l(\zeta)$ the longitude containing $\zeta$ and by $l(\zeta,M)$ the point of $l(\zeta)$ corresponding to a ray of slope $M$.
So \(l(\zeta,\pm\infty)\) are the poles and \(l(\zeta,0)\) lies in the equator.

\begin{example}[Twisting]
\label{ex:twisting}
Let $a$ and $b$ generate $\F_2$ and $c$ denote the generator of $\Z$.
Let $\Gamma$ denote the Cayley graph of $\F_2$, an infinite 4-valent tree with $v\in\Gamma$ a vertex.
Then we have a natural product action of $\F_2\times\Z$ on \(X=\Gamma\times\R\) by letting $c$ translate
$\R$ one unit.  We will denote the action of an element $g$ on a point $x\in X$ by $g\cdot x$.
In coordinates, we have
\begin{align*}
	a\cdot(v,t) &= (av,t), \\
	b\cdot(v,t) &= (bv,t), \\
\textrm{and }
	c\cdot(v,t) &= (v,t+1).
\end{align*}
Now we consider another action on the same space by ``twisting'' one of the generators.  Specifically, we change
the action by letting the action of $b$ on the $\R$-coordinate be the same translation as $c$.  We denote the resulting
action by $\ast$.  So we now have
\begin{align*}
	a\ast(v,t) &= (av,t), \\
	b\ast(v,t) &= (bv,t+1), \\
\textrm{and }
	c\ast(v,t) &= (v,t+1).
\end{align*}

Now if the identity map $G\to G$ is to extend to a continuous function between these two boundaries,
it must be the case that if a sequnce of group elements converges to a point of $\partial X$ in the
first action, then it also converges to a point of $\partial X$ in the second action.
We claim that this is not the case.

Our first observation is that in $\Gamma$, the sequences
\((a^nv)\) and \((a^nb^{n^2}v)\) converge to the same point of \(\partial\Gamma\);
we will call this point $a^\infty$.  In $X$ it follows that the sequences \((a^n\cdot x)\) and \((a^nb^{n^2}\cdot x)\)
have their limit point in \(l(a^\infty)\).  In fact they converge to the same point --
since they both only act in the $\Gamma$-coordinate, they both converge to $l(a^\infty,0)$.

Now consider the second action $G\ast X$.
The sequence \((a^n\ast x_0)\) also converges to $l(a^\infty,0)$.  But this time
\((a^nb^{n^2}\ast x_0)\) converges to a ray of slope 1.
For, the slope of the line segment \([x_0,a^nb^{n^2}\ast x_0]\) is
\[
	\frac{n^2}{n+n^2}\to 1.
\]
This means that the sequences \((a^n)\) and \((a^nb^{n^2})\) converge to the same boundary point in one action and two different
points in the other action.  So there is no limset map $\partial X\to\partial X$ which respects the group actions.

Given $\zeta\in\partial X$, let $\Lambda(\zeta)$ denote the subset of $\partial X$ consisting of points $\zeta'$ for which there
is a sequence $(g_n)$ such that $g_n\cdot x\to\zeta$ and $g_n\ast x\to\zeta'$.  Then $\Lambda(a^\infty)$ 
is an interval in the longitude of $a^\infty$ stretching between $l(a^\infty,1)$ and
$l(a^\infty,-1)$.
The sequences in $\F_2$ converging to $a^\infty$ in the action $G\cdot X$ are those
of the form $a^{k_n}w_n$ where $k_n\to\infty$.  But under the action $G\ast X$, such sequences have varying slope depending
on the asymptotic ratio of $b$'s to $a$'s.

Finally we note that the image of the map $\F_2\to X$ under the twisted action is not quasi-convex,
even though it is a QIE.  To see that the $\F_2$-orbit of a point
is not quasi-convex under $\ast$, observe that the geodesic $[b^n\ast x_0,ab^n\ast x_0]$ strays farther and farther from $x_0$ as $n$
gets large.  But it is a QIE because
\[
	l(w)\le d_X(x_0,w\ast x_0)\le\sqrt2l(w)
\]
where $w\in \F_2$ and $l$ denotes wordlength in $\F_2$ with respect to the generating set $\{a,b\}$.
This stands in contrast to the situation for $\delta$-hyperbolic spaces, where the concepts of quasi-convexity and
QIE are equiavlent \cite[Corollary III.$\Gamma$.3.6]{bridson-haefliger}.
\end{example}

\begin{example}[Stretching]
\label{ex:stretching}
Here is a different type of phenomenon.  Let $\Gamma'$ be a homeomorphic copy of $\Gamma$
in which the edges corresponding to $a$ have length 1 (as before) and edges corresponding to $b$ have length 2.  Let
$X'=\Gamma'\times\R$ and $G$ act on $X'$ via the product action.  We will denote this action by $\odot$.
Let \(x_0'\in X'\) be the preferred basepoint \((y_0',0)\) where
$y_0'\in\Gamma'$ is the vertex corresponding to $y_0$.  Since \(\F_2\) acts only in the $\Gamma'$-coordinate, the quasi-isometry
\(\F_2\cdot x_0\to \F_2\odot x_0'\) extends to a map between the equators.
However the quasi-isometry
\[
	(\F_2\times\Z)\cdot x_0\to (\F_2\times\Z)\odot x_0'
\]
does not extend.  In $X$, the sequences \((a^nc^n\cdot x_0)\) and \((a^nb^{n^2}c^{n^2}\cdot x_0)\) both converge to
\(l(a^\infty,1)\).
In $X'$, \((a^nc^n\odot x_0')\) also converges to $l(a^\infty,1)$, but \((a^nb^{n^2}c^{n^2}\odot x_0)\) converges to $l(a^\infty,1/2)$.
With $\zeta=l(a^\infty,1)$, define $\Lambda'(\zeta)$ as before, but replacing $G\ast X$ with $G\odot X'$.  Then again we see
that $\Lambda'(\zeta)$ is an interval in $\partial X'$.
\end{example}

In fact, it follows from Theorem \ref{Thm:CE} that whenever $\F_2\times\Z$ acts geometrically on two CAT(0) spaces $X$ and $X'$ and
$\zeta\in\partial X$, then $\Lambda(\zeta)$ is either a point or an interval.
In the above examples, one moves between the endpoints of this interval by controlling the ratio of $b$'s to $a$'s in the sequence
$(g_n)$.  Our final example shows that moving between the endpoints of this interval need not be so simple in general.

\begin{example}[Diamonds]
\label{ex:diamonds}
Begin with the diamond in $\E^2$ with vertices \(\{(1,0),(0,-1),\allowbreak (-1,0),(0,1)\}\) (that is, the convex hull of these four points).
Let $\overline Q$ be the space formed by gluing opposite vertices: \((0,1)\) to \((0,-1)\) and \((1,0)\) to \((-1,0)\).
Then \(\pi_1(\overline Q)=\F_2\).  Its universal cover $Q$ consists of diamonds glued vertex to vertex with its fundamental group,
$\F_2$, acting geometrically by deck transformations.  The generators of $\F_2$ are represented by the 2 geodesic loops based at
$(0,0)$ of length 2 which pass through the glued points.

We will use the symbol $\diamond$ for our action.
Let \(q_0\in Q\) be a preimage of the origin in $\overline Q$.  In this space, $a$ and $b$ both translate
$q_0$ a distance of 2.  But the geodesic \([q_0,ab\diamond q_0]\) is shorter than expected.
More generally:
\[
	d(q_0,(ab)^nq_0)
=
	(2n-1)\sqrt2+2
\sim
	n2\sqrt2
\]

\begin{figure}[ht!]
\begin{center}

\begin{tikzpicture}
\begin{scope}[scale=1.3]
\foreach \a in {(0,0),(2,0),(2,-2)}
\draw[black, thick] \a+(1,0)--+(0,1)--+(-1,0)--+(0,-1)--cycle;

\foreach \a in {(0,1.5),(2,1.5),(-1.5,0),(0,-1.5),(3.5,0),(3.5,-2)}
\draw[black, thick] \a+(1/2,0)--+(0,1/2)--+(-1/2,0)--+(0,-1/2)--cycle;

\foreach \a in {(3.5,-2.5-1/2)}
\draw[black, thick] \a+(1/2,0)--+(0,1/2)--+(-1/2,0)--+(0,-1/2)--cycle;

\node at (0,0)           {$1$};
\node at (2,0)           {$a$};
\node at (2,-2)          {$ab$};
\node at (0,1.5)         {$\scriptstyle{b^{-1}}$};
\node at (2,1.5)         {$\scriptstyle{ab^{-1}}$};
\node at (-1.5,0)        {$\scriptstyle{a^{-1}}$};
\node at (0,-1.5)        {$\scriptstyle{b}$};
\node at (3.5,0)         {$\scriptstyle{a^2}$};
\node at (3.5,-2)        {$\scriptstyle{aba}$};
\node at (3.5,-2.5-1/2)  {$\scriptstyle{(ab)^2}$};

\end{scope}
\end{tikzpicture}

\end{center}
\caption{$Q$}
\end{figure}

Let $G$ act on \(X''=Q\times\R\) via the corresponding product action, and let \(x_0''=(q_0,0)\).  In $G\cdot X$, the sequences \((a^nc^n\cdot x_0)\),
\((a^nb^{n^2}c^{n^2}\cdot x_0)\), and \((a^n(ab)^{n^2}c^{2n^2}\cdot x_0)\) all converge to $l(a^\infty,1)$.
But in $X''$, the sequences \((a^nc^n\diamond x_0'')\) and \((a^nb^{n^2}c^{n^2}\diamond x_0'')\)
converge to $l(a^\infty,1/2)$ whereas the sequence \((a^n(ab)^{n^2}c^{2n^2}\diamond x_0'')\) converges to $l(a^\infty,1/\sqrt2)$.  Again
\(\Lambda(a^\infty)\subset\partial X''\) is an interval.  But we do not
move between the endpoints by controlling the ratio of $a$'s to $b$'s -- we move by controlling the number of subwords of the form
$ab$ which appear.  In other words, it is by ``shuffling'' the $a$'s and $b$'s that we maximize slope.
\end{example}

\subsection{Examples of Schmears}
\label{examples of schmears}
Assume $G$ acts geometrically on two CAT(0) spaces $X_1$ and $X_2$ with basepoints $x_1\in X_1$ and $x_2\in X_2$.
Take $X=X_1\times X_2$ and $G^\Delta$ to be the ``diagonal subgroup'' described above.
Given a group element $g\in G$, we denote the slope of the line segment $[x_0,g^\Delta x_0]$
in this splitting by $\slopescal(g)$.
Then $\partial X$ is the join of $\partial X_1$ with $\partial X_2$.  $\Lambda=\limset(X,G^\Delta)$ lives in the
interiors of the join arcs.  The schmear maps $\Lambda\to\partial X_i$
come from collapsing the join arcs onto $\partial X_i$.

Interestingly, even though all boundaries of a negatively curved group $G$ are equivalent and the schmears
of these boundaries are typically homeomorphic to $\partial G\times I$ for some interval $I$, these schmears
need not be equivalent themselves.  In fact there need not exist limset maps going in either direction, as
exhibited here:

\begin{example}[Two Schmears for Boundaries of $\F_2$]
\label{ex:schmear}
Using the notation given in Example \ref{ex:stretching} (the stretching example), take \(G=\F_2\), \(X_1=\Gamma\), and \(X_2=\Gamma'\).
Then the schmear of the pair $\partial X_1,\partial X_2$ is $\Lambda_\odot=\partial\F_2\times[1,2]$.
For instance,
\[
	\lim_{n\to\infty}\slopescal_\odot(a^n)=1
\]
whereas
\[
	\lim_{n\to\infty}\slopescal_\odot(a^nb^{n^2})=2,
\]
even though in $\F_2\cup\partial\F_2$, both sequences converge to $a^\infty$.  The fact that $\F_2$ is strongly rigid means that
the maps \(\Lambda_\odot\to\partial X_i\) simply collapse the interval factors.

Now take $G=\F_2$ and $X_1=\Gamma$ but $X_2=Q$ from Example \ref{ex:diamonds} (the diamonds example).
Here the schmear is $\Lambda_\diamond=\partial\F_2\times[1/\sqrt2,1/2]$.  It is homeomorphic to $\Lambda_\odot$,
but not equivalent to it.  In fact there are no limset maps going in either direction.
For, in $ G^\Delta\cup\Lambda_\odot$, the sequences $(a^n)^\Delta$ and $(a^nb^{n^2})^\Delta$ converge
to different points, whereas in $ G^\Delta\cup\Lambda_\diamond$ they converge to the same point:
\[
	\lim_{n\to\infty}\slopescal_\diamond(a^n)=\lim_{n\to\infty}\slopescal_\diamond(a^nb^{n^2})=\frac12.
\]
On the other hand, the pair of sequences $(a^{n+n^2}b^{n^2})^\Delta$ and $(a^n(ab)^{n^2})^\Delta$
converge to different points of $\Lambda_\diamond$; the first has slope $1/2$, the second $1/\sqrt2$.
But they converge to the same point of $\Lambda_\odot$, since they are written using the exact same letters
and shuffling of letters does not change $\slopescal_\odot$.
\end{example}

There is a bound on the Lebesgue convering dimension of a schmear in terms of the covering dimension of a boundary.
Extending work of Bestvina \cite{bestvina-local_homology}, Geoghegan and Ontaneda have proven that the
dimension of a group boundary is one less than the cohomological dimension of the group \cite{geoghegan-ontaneda}.
So if the cohomological dimension of a CAT(0) group $G$ is $d+1$ (so that all of its boundaries have dimension $d$),
then whenever $\Lambda$ is a schmear for a pair of boundaries $\partial X$, $\partial Y$, it lives in the space
$\partial X\times\partial Y\times(0,\infty)$, which has dimension bounded by $2d+1$.

The schmear maps $\phi_i:\Lambda\to\partial X_i$
can be taken to be the projection maps to the respective coordinates.
So point preimages live in subspaces of the form
\begin{align}
	\{\zeta_1\}\times\partial X_2\times(0,\infty)
\label{eqn:ambient}
\end{align}
which has dimension $d+1$.
For $\F_2\times\Z$, the bound on the dimension of the schmear is 3,
and the preimage of every point has dimension no more than 2.
The schmears of the pairs of boundaries for $\F_2\times\Z$ described in the previous
subsection all have dimension 2.
The following schmear has full dimension.

\begin{example}[A 3-Dimensional Schmear for Boundaries of $\F_2\times\Z$]
\label{ex:3 dimensional schmear}
Let $\cdot$ denote the untwisted product action of $\F_2\times\Z$ on $X_1=\Gamma\times\R$
first described in Example \ref{ex:twisting}
with $x_1\in X_1$ the preferred basepoint.
Let also $\diamond$ denote the action of $\F_2$ on $Q$ first described in Example \ref{ex:diamonds}.
Take $X_2=Q\times\R$, $v=(0,0)$, and $x_2=(v,0)$ but with the following action by
$\F_2\times\Z$ obtained by adding a twist to the product action:
\begin{align*}
	a\star(v,t) &= (a\diamond v,t), \\
	b\star(v,t) &= (b\diamond v,t+1), \\
\textrm{and }
	c\star(v,t) &= (v,t+1).
\end{align*}
Here are four sequences which converge to $a^{\infty}\in\partial X_1$ but which ``fan out'' in both the
$\partial X_2$ direction and the $(0,\infty)$ direction of the ambient space containing the schmear.
\[
\begin{array}{|l|l|l|}
\hline
	(g_n) & \lim_{n\to\infty}(g_n\star x_2) & \lim_{n\to\infty}\slopescal(g_n) \\
\hline
\hline
	x_n = a^nb^{-n^2} & l(a^\infty,-\frac12) & \sqrt5 \\ 
\hline
	y_n = a^nb^{n^2} & l(a^\infty,\frac12) & \sqrt5 \\ 
\hline
	x_n' = a^n(ab^{-1})^{n^2} & l\left(a^\infty,-\frac{1}{\sqrt2}\right) & 3 \\ 
\hline
	y_n' = a^n(ab)^{n^2} & l\left(a^\infty,\frac{1}{\sqrt2}\right) & 3 \\ 
\hline
\end{array}
\]
It follows from Theorem \ref{Thm:CE} that the preimage of $a^{\infty}$ under the schmear map
$\Lambda\to\partial X_1$ contains the convex hull of these four points and is a 2-disk $D^2$.
The entire schmear is homeomorphic to the suspension of $C\times D^2$
\end{example}

\section{Preliminaries}
A \textit{compactification} of a Hausdorff space $X$ is a topological embedding $X\incl\overline X$
into a compact Hausdorff space whose image is dense.  We will refer to $\overline X-X$ as a \textit{limset} for $X$.
Let $\Lambda,\Lambda'$ be limsets of a noncompact space $X$.  We write $\Lambda\ge\Lambda'$
if the identity function $X\to X$ extends continuously to a map $X\cup\Lambda\to X\cup\Lambda'$.
The restriction of this map to $\Lambda\to\Lambda'$ is called a \textit{limset map}.
As mentioned above, if this map is a homeomorphism, then $\Lambda$ and $\Lambda'$ are considered equivalent.
Let $X$ be a noncompact space and $\eL(X)$ denote the collection of equivalence classes of limsets.

\begin{rem}
The reader should note that homeomorphic or even isometric limsets may represent different elements of $\eL(X)$.
When we refer to $\Lambda$ as an element of $\eL(X)$, a particular compactification
(informally, the way $\Lambda$ is glued to $X$) is assumed.
\end{rem}

The following statements are easily verified.

\begin{lemma}
\label{le:characterization of existence of limset maps}
Let $X$ be a noncompact Hausdorff space $X$.
For $\Lambda,\Lambda'\in\eL(X)$, $\Lambda\ge\Lambda'$ iff the following rule holds:
Given any sequence $(x_n)\subset X$ converging to a point of $\Lambda$, $(x_n)$ also converges
to a point of $\Lambda'$.
\end{lemma}

\begin{lemma}
Let $X$ be a noncompact Hausdorff space.  Then
\begin{enumerate}
\item Limset maps between elements of $\eL(X)$ are unique and surjective (when they exist)\\
\item $\ge$ is a partial ordering on $\eL(X)$\\
\item Every diagram of limset maps commutes\\
\end{enumerate}
\end{lemma}

\begin{rem}
If $X$ is a noncompact Hausdorff space, then $\eL(X)$ has a unique maximum,
namely $\beta X-X$ where $\beta X$ is the Stone-\v{C}ech compactification of $X$.
\end{rem}

\begin{lemma}
\label{prop:rationalmaps}
Let \(Y\subset X\) be a subspace, $\Lambda\ge\Lambda'$ be limsets of $X$.
Let $\Lambda(Y)=\overline Y\cap\Lambda$ and $\Lambda'(Y)=\overline Y\cap\Lambda'$.
Then $\Lambda(Y)\ge\Lambda'(Y)$ via the restriction of the limset map
$\Lambda\to\Lambda'$ to $\Lambda(Y)$.
\end{lemma}

\subsection{Notes on CAT(0) Spaces and Groups}
Let $(X,d)$ be a CAT(0) space.
Given two geodesics $\alpha,\beta:[0,1]\to X$ parameterized to have constant speed, the distance metric
satisfies
\[
	d\bigl(\alpha(t),\beta(t)\bigr)
		\le
	td\bigl(\alpha(1),\beta(1)\bigr)
		+
	(1-t)d\bigl(\alpha(0),\beta(0)\bigr).
\]
for all $t\in[0,1]$.
This property is known as \textit{convexity of the metric}.

For the most part we will consider the visual boundary as the
set of all geodesic rays emanating from a common basepoint $x_0\in X$.  With the cone topology,
$\overline X=X\cup\partial X$ is a compactification of $X$.  $\overline X$ can be identified with the
space of geodesic segments and rays emanating from $x_0$ parameterized to have unit speed.
Then the topology on $\overline X$ is the same as the compact-open topology on this function space.

Here is how we can tell if a sequence of points $(x_n)$ in $\overline X$ converges to a point $\zeta\in\partial X$.
Let $\gamma_n$ and $\alpha$ be the geodesics based at $x_0$ determining $x_n$ and $\zeta$ respectively, and
$\delta>0$ be arbitrary.  Then $x_n\to\zeta$ iff for any $K\ge 0$, there is an $N\ge 0$ such that for all
$n\ge N$, either $x_n\in\partial X$ or $d(x_0,x_n)\ge N$, and $d(\alpha(K),\gamma_n(K))\le\delta$.
Convexity of the metric then guarantees that $\gamma_n\to\alpha$ uniformly on compact subsets of $\R$.

Let $G$ be a group acting on a proper CAT(0) space $X$ by isometries.
The \textit{translation length} of $g\in G$ is defined by
\( |g|=\inf_{x\in X}d(x,gx) \).
If this value is realized then we call $g$ \textit{semi-simple}.
Proposition II.6.10 in \cite{bridson-haefliger} tells us that whenever a group acts
geometrically on a proper CAT(0) space $X$, all of its elements are semi-simple.
The set of all points on which this minimum is attained is called the \textit{minset} of $g$:
\[
	\Min_X(g)=
	\Min(g)=\bigl\{x\in X\big|d(x,gx)=|g|\bigr\}.
\]
If $|g|>0$, then $g$ is called \textit{hyperbolic}.  Whenever $x$ is in the minset of
a hyperbolic element $g$, then there is a geodesic line $A$ passing through $x$ which is $g$
invariant.  This line is called an \textit{axis} of $g$.  For a subgroup $H\le G$ of isometries
of a CAT(0) space $X$ the \textit{minset of $H$} is defined as
\[
	\Min_X(H)=\Min(H)=\bigcap_{g\in H}\Min(g).
\]
By \cite[Proposition II.6.2]{bridson-haefliger}, minsets of isometries and minsets of groups of isometries
are always closed and convex.  If the ambient space is complete (as we will always assume), then
minsets are themselves complete CAT(0) spaces.

A key ingredient in Theorem \ref{Thm:CE} is the Flat Torus Theorem \cite[Theorem II.7.1]{bridson-haefliger}.
Recall that an isometry $g$ of a product $X_1\times X_2$ is said to \textit{respect the product decomposition}
if $g$ can be written in coordinates as $g=(g_1,g_2)$ where $g_1$ and $g_2$ are isometries of $X_1$ and $X_2$.

\begin{Theorem}[Flat Torus Theorem]
Let $A=\Z^d$ act properly discontinuously by semi-simple isometries on a proper CAT(0) space $X$.
Then:\\
\begin{enumerate}
\item $\Min(A)$ is nonempty and splits as a product $Y\times\E^d$.\\
\item Every $c\in A$ leaves $\Min( A)$ invariant and respects the product decomposition;
$c$ acts as the identity on $Y$ and as a translation on $\E^d$.\\
\item The quotient of each flat $\{y\}\times\E^d$ by the action of $ A$ is an $n$-torus.\\
\item If an isometry of $X$ normalizes $ A$, then it leaves $\Min(A)$ invariant and preserves
the product decomposition.\\
\end{enumerate}
\end{Theorem}

Another fact about minsets which will come in handy is \cite[Proposition II.6.9]{bridson-haefliger}:

\begin{prop}
\label{prop:minsetproduct}
Let $X_1$ and $X_2$ be proper CAT(0) spaces and $g_i$ be isometries
of $X_i$, and consider the isometry $g=(g_1,g_2)$ of $X=X_1\times X_2$.
Then
\[
	\Min_Xg=\Min_{X_1}(g_1)\times\Min_{X_2}(g_2)
\]
In particular, $g$ is semi-simple iff both $g_1$ and $g_2$ are.
\end{prop}

\subsection{A Category of CAT(0) Limsets}
If a group $G$ acts by isometries on a CAT(0) space $X$ in such a way that $G\to X$ is a QIE,
we will say that $G$ acts pseudo-geometrically on $X$.
The key way in which psuedo-geometric actions arise in this paper
is as follows: Whenever a group $G$ acts geometrically on a CAT(0) space $X$ and $H\le G$ is a
quasi-isometrically embedded subgroup,
then the action of $H$ as a subgroup is pseudo-geometric.

We denote by $\eLcat(G)$ the subcollection of $\eL(G)$ consisting of limsets which come from pseudo-geometric actions.
Formally,
\[
	\eLcat(G)
=
	\bigl\{
		\limset(X,G)
	\big|
		G\textrm{ acts pseudo-geometrically on a CAT(0) space }X
	\bigr\}
		/\sim
\]
where $\sim$ denotes the equivalence relation described above.

\begin{rem}
By Theorem \ref{Thm:Schmear}, $\eLcat(G)$ is a directed poset which means that it has the structure of an inverse system.
\end{rem}

\begin{example}[Limset Maps for $\Z^d$]
Whenever $\Z^d$ acts properly discontinuously by semi-simple isometries on a CAT(0) space $X$,
the Flat Torus Theorem guarantees that it acts cocompactly on
the convex hull of the orbit of a point.  The point can be chosen so that this convex hull is $\E^d$,
and $\limset(X,\Z^d)=\partial\E^d=S^{d-1}$.  It follows that all limsets in $\eLcat(\Z^d)$ which come
from semi-simple isometries are equivalent.
\end{example}

This line of reasoning also gives

\begin{corollary}[Corollary to the Flat Torus Theorem]
\label{co:FTT}
Let $\Z^d$ act by semi-simple isometries on two CAT(0) spaces $X_1$ and $X_2$.
Then there is a limset map $\limset(X_1,\Z^d)\to\limset(X_2,\Z^d)$ which is a homeomorphism.
\end{corollary}

\begin{lemma}
\label{lemma:subgroups}
Let $\Lambda\ge\Lambda'\in\eLcat(G)$ and $H\le G$ be quasi-isometrically embedded subgroup.
Denote by $\Lambda(H),\Lambda'(H)$ the corresponding subsets which are the limsets of $H$.
Then the limset map $\phi:\Lambda\to\Lambda'$ restricts to a limset map $\phi_H:\Lambda(H)\to\Lambda'(H)$.
\end{lemma}

In general, the preimage of $\Lambda'(H)$ under the unrestricted map $\phi$ can be larger than $\Lambda(H)$.

\begin{example}
\label{ex:limsetproblem}
Consider Example \ref{ex:twisting}.  We have $\F_2$ acting on \(\Gamma\times\R\) in two ways;
let $X_1$ denote $\Gamma\times\R$ with the twisted action, $X_2=\Gamma$ with the standard action, and \(H=\left<a\right>\).
There is a limset map \(\phi:\limset(X_1,\F_2)\to\limset(X_2,\F_2)\).  But \(\limset(X_1,H)\) is just the two points
\(l(\{a^{\pm\infty}\},0)\) whereas \(\phi^{-1}(\limset(X_2,H))\) consists of two intervals.
\end{example}

All CAT(0) boundaries of $\F_2$ are equivalent in $\eLcat(\F_2)$, since it is negatively curved.
They are homeomorphic to the Cantor set $C$.
But unless the actions chosen on the spaces $X_1$ and $X_2$ are very close, the schmears constructed
will be homeomorphic to $C\times[0,1]$.
But as mentioned in Example \ref{ex:schmear}, these schmears need not themselves even be comparable
in $\eLcat(\F_2)$.

The following proposition tells us that all of the major results in this paper hold if they hold for a subgroup
of finite index.

\begin{prop}
\label{prop:finiteindex}
Let $H$ be a finite index subgroup of a CAT(0) group $G$.  Assume $G$ acts pseudo-geometrically on CAT(0)
spaces $X$ and $Y$ with limsets $\Lambda$ and $\Lambda'$ respectively.
Then $\Lambda=\limset(X,H)$, $\Lambda'=\limset(Y,H)$, and $\Lambda\ge\Lambda'$ in $\eLcat(G)$ iff
$\Lambda\ge\Lambda'$ in $\eLcat(H)$.  If so, then the limset map $\Lambda\to\Lambda'$ which comes from
$G$ is the same as that which comes from $H$.
\end{prop}

\begin{proof}
Choose basepoints $x\in X$ and $y\in Y$.
Since $H$ has finite index in $G$, inclusion $H\to G$ is a quasi-isometry, and hence $H\to X$ and $H\to Y$
are QIEs.  Let $K\ge 0$ be a constant such that given $g\in G$, there is an $h\in H$ such that
$d_X(gx,hx)$ and $d_Y(gy,hy)$ are both bounded by $K$.  Then in particular, if $(g_n)\subset G$ converges to a point
of $\partial X$, there is a sequence $(h_n)\subset H$ such that $d_X(g_nx,h_nx)\le K$, so that $(h_nx)$ converges
to the same point of $\partial X$.  Since $d_Y(g_ny,h_ny)\le K$ as well, we know that the sequence $(g_ny)$
converges iff $(h_ny)$ converges, and if they do, then they converge to the same point of $\partial Y$.
\end{proof}

\section{The Schmear}
\label{sec:schmear}
We begin this section with a more formal treatment of slopes in products of CAT(0) spaces.
Let $X=X_1\times X_2$ be a product of two proper CAT(0) spaces.
A \textit{path} is a map of an interval \(I\subset\R\) into a space, where $I$ may either
be a closed interval $[0,D]$ or a half-open interval $[0,\infty)$.
Given a path $\alpha:I\to X$, we will write \(\alpha=(\alpha_1,\alpha_2)\)
to mean that for every $t$ in the domain of $\alpha$,
\(
	\alpha(t)=\bigl(\alpha_1(t),\alpha_2(t)\bigr).
\)
For us a path $\alpha$ will be called a \textit{geodesic} if there is a number \(\sigma\ge 0\) such that for every
$s\neq t$ in the domain of $\alpha$, we have
\[
	\frac
		{d\bigl(\alpha(s),\alpha(t)\bigr)}
		{|s-t|}
			=
		\sigma.
\]
$\sigma$ is called the \textit{speed} of $\alpha$.  This may mean that $\alpha$ is a \textit{geodesic
segment} (a geodesic with domain of the form $[0,D]$), a \textit{geodesic ray} (a geodesic with domain of the form 
$[0,\infty)$), or even perhaps a \textit{constant geodesic} (a geodesic whose image is a point). 
If $\sigma=1$, we say that the geodesic has \textit{unit speed}.

Let $\alpha=(\alpha_1,\alpha_2)$ be a nonconstant geodesic in $X$.  We define the \textit{slope of $\alpha$} by
\[
	\slopescal(\alpha)
=
	\frac
		{d_2\bigl(\alpha_2(s),\alpha_2(t)\bigr)}
		{d_1\bigl(\alpha_1(s),\alpha_1(t)\bigr)}
\]
for some $s\neq t$ in the domain of $\alpha$.  By the following lemma $M(\alpha)$ is well-defined
and independent of choice of parameterization.  As before, we allow $\infty$ as a value for $\slopescal$.

\begin{lemma}
\label{le:coordinategeospeeds}
Let \(\alpha=(\alpha_1,\alpha_2)\) be a path in $X_1\times X_2$.  Then $\alpha$ is a geodesic iff
$\alpha_1$ and $\alpha_2$ are.
If $\sigma_1$ and $\sigma_2$ are the speeds of $\alpha_1$ and $\alpha_2$,
then $\slopescal(\alpha)=\sigma_2/\sigma_1$ and the speed of $\alpha$ is $\sqrt{\sigma_1^2+\sigma_2^2}$.
\end{lemma}

\begin{proof}
The fact that $\alpha$ is a geodesic iff each $\alpha_i$ is \cite[Proposition I.5.3(3)]{bridson-haefliger}.
Once this is established, we know that every geodesic $\alpha$ lives in a subspace isometric to a flat
quadrant or square, where the equations relating speeds and slope are easily verified.
\end{proof}

In the current setup if \(\alpha_1\) is constant (its image is just a point),
then it is natural to say that \(\slopescal(\alpha)=\infty\).  If we fix a basepoint \(x_0\in X\) and identify points of
\( X\cup\partial  X\) with the collection of unit speed geodesics emanating from $x_0$, then this gives a function
\[
	\slopescal: (X-\{x_0\})\cup\partial X\to[0,\infty].
\]

\begin{prop}
\label{prop:Miscontinuous}
$\slopescal$ is continuous.
\end{prop}

\begin{proof}
Fix \(\epsilon>0\) and let $\pi_\epsilon$ denote the geodesic retraction based at $x_0$
\[
	X\cup\partial X-B_\epsilon(x_0)\to S_\epsilon(x_0)
\]
where $S_\epsilon$ denotes the sphere of radius $\epsilon$;
this is a continuous function by convexity of the metric.  Let $\rho_i$ denote coordinate projection
\(X\to X_i\), which is also continuous.  Finally, let \(\delta_i(x)=d_i(\rho_i(x_i),\rho_i(x_0))\).
Then on \(X\cup\partial X-B_\epsilon(x_0)\),
\[
	\slopescal
=
	\frac
		{\delta_2\circ\pi_\epsilon}
		{\delta_1\circ\pi_\epsilon}.
\]
Therefore $\slopescal$ is continuous on \(X\cup\partial X-B_\epsilon(x_0)\) for every \(\epsilon>0\).
\end{proof}

If $G$ acts properly discontinuously on $X$, then
for \(g\in G\) not stabilizing $x_0$, we may define \(\slopescal(g)=\slopescal([x_0,gx_0])\).
Since the stabilizer of $x_0$ is finite and we are interested in what happens at infinity, we will make the arbitrary definition
$\slopescal(g)=0$ when $gx_0=x_0$.  $\slopescal$ extends continuously to \(\Lambda=\limset( X,G)\).

\subsection{The Construction}
\label{subsection:intermediatelimset}
This subsection completes the proof of Theorem \ref{Thm:Schmear}.
Assume that $G$ acts pseudo-geometrically on two CAT(0) spaces $X_1$ and $X_2$.
Consider the inclusion \(G\incl G\times G\) as the diagonal subgroup \(G^\Delta=\{(g,g)|g\in G\}\).
We define
\[
	\Lambda=\Lambda(G^\Delta,X_1,X_2)=\limset( X,G^\Delta).
\]
Choose basepoints \(x_i\in X_i\) and set \( x=(x_1,x_2)\in X=X_1\times X_2\).

\begin{lemma}
The map $G^\Delta\to X$ is a QIE.
If in addition the action of $G$ on each $X_i$ is semi-simple, then
so is the action of $G^\Delta$ on $X$.
\end{lemma}

\begin{proof}
Since $G\times G$ acts on $X$ geometricaly, we know that $G\times G\to X$ is a QIE by
the \v{S}varc-Milnor Lemma.
To see that $G^\Delta\to X$ is a QIE it suffices to check that the isomorphic embedding
$G\to G\times G$ which takes $G$ to $G^\Delta$ is a QIE.  Let $l$ be a length
metric on $G$ with respect to some finite generating set $\eS$.  Let $\eS_1=\eS\times\{1\}$
and $\eS_2=\{1\}\times\eS$ which generate $G\times\{1\}$ and $\{1\}\times G$ respectively.
Then $\eS_1\cup\eS_2$ is a finite generating set for $G\times G$ inducing the length metric
$l'(g,h)=l(g)+l(h)$.  In particular, $l'(g,g)=2l(g)$.
To get the second statement, simply apply Proposition \ref{prop:minsetproduct}.
\end{proof}

Assume $G$ acts pseudo-geometrically on CAT(0) spaces $X_1$ and $X_2$ and consider the
action of $G^\Delta$ on $X=X_1\times X_2$ with $\Lambda=\limset(X,G^\Delta)$.

\begin{lemma}
\label{lemma:limDGbdd}
\(\Lambda\) misses \(\partial X_1\cup\partial X_2\).
\end{lemma}

\begin{proof}
We use the fact that the map \(Gx_1\to Gx_2\) is a QIE
to prove that \(\slopescal(\Lambda)\) is bounded away from 0 and $\infty$.
Let $\lambda\ge 1$ and $\epsilon\ge 0$ be constants such that for \(g\in G\),
\[
	\frac{1}{\lambda}d_1(x_1,gx_1)-\epsilon
		\le
	d_2(x_2,gx_2)
		\le
	\lambda d_1(x_1,gx_1)+\epsilon.
\]
Then whenever we have a sequence \((g_n)\subset G\) such that
\(g_n^\Delta x\to\zeta\in\partial X\), it follows that $d_1(x_1,g_nx_1)\to\infty$.
So
\[
	\frac
		{\frac{1}{\lambda} d_1(x_1,g_nx_1)-\epsilon}
		{d_1(x_1,g_nx_1)}
		\le
	\slopescal(g_n^\Delta)
		\le
	\frac
		{\lambda d_1(x_1,g_nx_1)+\epsilon}
		{d_1(x_1,g_nx_1)}.
\]
Letting \(n\to\infty\), we get
\[
	\frac1\lambda\le \slopescal(\zeta)\le\lambda.
\]
\end{proof}

Consider the maps $e_i:\partial X-(\partial X_1\cup\partial X_2)\to\partial X_i$ which collapse the join arcs.
Specifically, take a ray $\alpha=(\alpha_1,\alpha_2)$ in $X$ based at $x$.  To say that $\alpha(\infty)$ lies interior
to a join arc means that $\slopescal(\alpha)$ is not 0 or $\infty$.  Then $e_i(\alpha(\infty))=\alpha_i(\infty)$.
This extends the projection map \(X\to X_i\) to an (open) subset of $\partial X$.
We take $\phi_i=e_i|_\Lambda$.  The following lemma
completes the proof of Theorem \ref{Thm:Schmear}:

\begin{lemma}
$\phi_1$ and $\phi_2$ are limset maps.
\end{lemma}

\begin{proof}
Assume we have \((g_n)\subset G\) such that $g_n^\Delta x\to\zeta\in\partial X$.
Let $\alpha=(\alpha_1,\alpha_2)$ be the unit speed geodesic ray with $\alpha(\infty)=\zeta$.
Then \(\phi_i(\zeta)=\alpha_i(\infty)\).
Let $\gamma_n=(\gamma_n^1,\gamma_n^2)$ be a unit speed
parameterization of $[ x,g_n^\Delta x]$.  Since projections do not increase distance, we have for
all \(t\ge 0\) and $n$ large enough so that $d( x,g_n^\Delta x)\ge t$,
\(
	d\bigl(\gamma_n^i(t),\alpha_i(t)\bigr)
		\le
	d\bigl(\gamma_n(t),\alpha(t)\bigr).
\)
Therefore, \(\gamma_n(t)\to\alpha(t)\) implies that \(\gamma_n^i(t)\to\alpha_i(t)\).
\end{proof}

Recent work by Link on lattices in certain CAT(0) groups allows us to understand schmears of negatively
curved groups \cite{link}.  Although Link's work applies to a wider class of CAT(0) groups, the following
argument requires strong rigidity, which is is only known for a handful of types of groups.

\begin{Co}
\label{Co:Link}
Let $G$ be a negatively curved group acting gometrically on two CAT(0) spaces $X_1$ and $X_2$,
and $\Lambda$ be the schmear of the (equivalent) pair $\partial X_1$,$\partial X_2$.  Then $\Lambda$ is homeomorphic to $\partial G\times I$
where $I$ is a closed possibly degenerate interval.
The limset maps $\Lambda\to\partial X_1$ simply collapse the $I$ coordinate.
\end{Co}

\begin{proof}
If $G$ is elementary, then $G^\Delta$ is an infinite virtually cyclic group acting by semi-simple isometries on
the CAT(0) space $X=X_1\times X_2$.  Its limit set $\Lambda$ is just two points, namely the endpoints of the axis
of the finite index infinite cyclic subgroup.  Assume $G$ is non-elementary.
Then, in the language of \cite{link}, it contains two independent regular axial isometries.
By Lemma \ref{lemma:limDGbdd}, we have $\Lambda\subset\partial X_1\times\partial X_2\times(0,\infty)$
and the maps $\phi_1$ and $\phi_2$ collapse the third coordinate.  If we reparameterize the third coordinate
according to angles made with $X_1$ instead of slopes, then $\Lambda\subset\partial X_1\times\partial X_2\times(0,\pi/2)$.

By \cite[Theorem B]{link}, $\Lambda$ splits as a product $F_G\times P_G$ where $F_G\subset\partial X_1\times\partial X_2$
and $P_G\subset(0,\pi/2)$.
Since $G$ is strongly rigid, the two boundaries $\partial X_1$ and $\partial X_2$ are equivalent.
By commutativity of the diagram
\[
\begin{array}{ccccc}
	&& \Lambda \\
	& \swarrow && \searrow \\
	\partial X_1 &&\approx&& \partial X_2
\end{array}
\]
it follows that $F_G$ is the graph of the homeomorphism.  By \cite[Theorem C]{link},
$P_G$ is a closed interval.
\end{proof}

Recall that the \textit{Tits metric} on the boundary of a CAT(0) space is the length metric associated
to the Tits angle metric and generates a finer topology than the cone topology.
In fact, there is a choice of $I\subset(0,\pi/2)$ such that when $\partial G$ is given the Tits metric
(giving it the topology of an uncountable discrete space) and $\Lambda$ is given the subspace metric
corresponding to the Tits metric on $\partial(X_1\times X_2)$ (giving it the structure of a spherical
join), then $I$ may be chosen so that the homeomorphism $\partial G\times I\to\Lambda$ is an
isometry with respect to this metric.

Link's results apply in a much more general context then the negatively curved setting.
For instance, the proof of this Corollary can be extended to groups $G$ acting geometrically
on CAT(0) spaces with isolated flats.
If they are not virtually abelian, then they have infinite Tits diameter by a theorem
of Hruska and Kleiner in \cite{hruska-kleiner}.  Then we apply a theorem of Papasoglu and Swenson
to find a pair of independent regular axial isometries \cite{papasoglu-swenson}.

\section{Geometric Actions of \texorpdfstring{$G=H\times\Z^d$}{G=H x Zd}}
\label{sec:mapsareCE}
Let $H$ be a CAT(0) group and $G=H\times\Z^d$.
We will write the coordinates of elements of $G$ in this direct product using the notation $\grouppair{w}{c}$.
Assume $G$ acts geometrically on a CAT(0) space $X$.
Applying the Flat Torus Theorem,
the minset $\Min(\Z^d)$ of $\Z^d$ is closed, convex, $G$-invariant, and splits as \(Y\times E\)
where $E$ is an isometric copy of $\E^n$.
The fact that $\Min(\Z^d)$ is convex and $G$-invariant
means that $\limset(X,G)\subset\partial\Min(\Z^d)$; for if we choose $x\in\Min(\Z^d)$, then $Gx\subset\Min(\Z^d)$.
We will assume that $X=Y\times E$.
Furthermore, the action of $G$ preserves this splitting in the following sense: Every isometry
\(g\in G\) can be written coordinate-wise as \((g_Y,g_E)\) where $g_Y$ and $g_E$ are isometries of $Y$ and $E$.
For clarity, the two notations are related by the rules $\grouppair{w}{c}_Y=w_Y$ and $\grouppair{w}{c}_E=w_Ec_E$.
The factor $\Z^d$ acts only in the $E$-coordinate; that is, for \(c\in\Z^d\), \(c_Y=\id_Y\).
The action of $\Z^d$ on $E$ is geometric which means that \(\partial E=\limset(X,\Z^d)\).
In addition, the projected action of $H$ on $Y$ by the rule $h\mapsto h_Y$ is also geometric.

Choose a basepoint $y_0\in Y$ and let \(x_0=(y_0,0)\) be the specified basepoint in $X$
where $0$ is the origin in $E$.

\begin{lemma}
The action of $H$ on $E$ is by translations.
\end{lemma}

\begin{proof}
Any isometry of $E$ which does not fix a point is a translation.
Choose any $w\in  H$.
By hypothesis, $\Z^d$ centralizes $w_E$, and by (4) from the Flat Torus Theorem,
$\Min w_E$ is $\Z^d$-invariant.
Since $\Min w_E$ is convex, it follows that \(\Min w_E=E\).
This means that if $w_E$ fixes a point, then \(w_E=\id_E\).
\end{proof}

\subsection{A Slope Vector Function}
For \(g\in G\), we define the \textit{vertical translation of \(g\) in $X$} to be the vector
\[
	\Vtrans(g)=g_E0\in E.
\]
We also define the \textit{horizontal displacement of $g$} by
\[
	\Hdisp(g)
		=
	d_Y(y_0,g_Yy_0).
\]
Since $\Z^d$ acts only in the $E$-coordinate, $\Hdisp(\grouppair{w}{c})=\Hdisp(w)$.
Let $\slopescal$ denote the slope function based at $x_0$ with respect to the splitting $Y\times E$
and write
\(
	\overline X_0
= 
	\slopescal^{-1}([0,\infty)). 
\)
This is just the complement of \(E \cup\partial E\) in $X$.
We define the \textit{slope vector map} \( \slopevec:\overline X_0\to E \)
as follows.  Identify $\overline X_0$ with the collection of nonvertical unit speed geodesics emanating from $x_0$.
For such a geodesic $\gamma=(\alpha,\beta)$ in the domain, we define
\[
	\slopevec(\gamma)
=
	\frac
		{\beta(t)-\beta(s)}
		{d_Y\bigl(\alpha(t),\alpha(s)\bigr)}
\]
where \(s<t\) are in the domain of $\alpha$.

\begin{lemma}
$\slopevec$ is independent of which $s$ and $t$ are chosen and is continuous.
\end{lemma}

\begin{proof}
Let $\sigma_\alpha$ and $\sigma_\beta$ denote the speeds of $\alpha$ and $\beta$.
Observe that
\[
	\beta(t)-\beta(s)=(t-s)\unitvec\sigma_\beta
\]
where $\unitvec\in E$ is unit vector in the direction of $\beta$.
Thus
\(
	\slopevec(\gamma)
=
	\sigma_\beta \unitvec/\sigma_\alpha,
\)
which is independent of $s$ and $t$.  The proof that $\slopevec$ is continuous is the same
as the proof of Proposition \ref{prop:Miscontinuous} but we replace $\delta_2$ with
\(\id_E\).
\end{proof}

Note that \(\|\slopevec\|\) is the same as the slope $\slopescal$ described earlier with respect to the current splitting
where $\|\cdot\|$ denotes the standard Euclidean norm on $E$.

\begin{lemma}
$\slopevec$ extends to a continuous map \((X-\{x_0\})\cup\partial X\to E\cup\partial E\).
Specifically, for \(v\in E\), \(\slopevec(y_0,v)\) is the geodesic ray
emanating from 0 passing through $v$,
and $\slopevec$ is the identity on $\partial E$.
\end{lemma}

\begin{proof}
For a point \(x\in X\cup\partial X\) not in $E\cup\partial E$,
define $\unitvec(x)=\slopevec(x)/\slopescal(x)$.  This is the unit vector in the direction of
$\slopevec(x)$.
If a sequence of points \((x_n)\subset\overline X_0\) converges to a point \(z\in(E-0)\cup\partial E\),
then eventually \((x_n)\) misses the subspace $Y\cup\partial Y$.
Then $\unitvec(x_n)\to \unitvec(z)$ and $\slopescal(x_n)\to\infty$.  So $\slopevec(x_n)=\slopescal(x_n)\unitvec(x_n)$ converges to the geodesic ray
in the direction of $\unitvec(z)$.
\end{proof}

As before, $\slopevec$ can be thought of as a continuous function of $G$ using the map \(G\to Gx_0\) provided
we make the arbitrary definition \(\slopevec(g)=0\in E\) for $g$ stabilizing $x_0$.  Note that for
$g=\grouppair{w}{c}\in G$ such that \(\Hdisp(g)>0\), we have
\[
	\slopevec(g)
=
	\frac{\Vtrans(g)}{\Hdisp(w)}
=
	\frac{\Vtrans(w)+\Vtrans(c)}{\Hdisp(w)}.
\]

Ruane proves in \cite{ruane-anglequestion} that whenever a CAT(0) group of the form $H\times\Z$
acts geometrically on a CAT(0) space, then the limset of $H$ is bounded away from the limset of
$\Z$.  Here is an analogous result in this setting.

\begin{prop}
$\slopevec$ is bounded on $H$.
\end{prop}

\begin{proof}
Choose any finite generating set for $H$ and let $l$ be the corresponding length metric.
Let \(\lambda\ge 1\) and \(\epsilon\ge 0\) be such that for all \(w\in  H\),
\(\Hdisp(w)\ge l(w)/\lambda-\epsilon \).
Let $M$ be the maximum of $\|\Vtrans\|$ on a finite
generating set for $ H$.  Then for $w\in H$, $\|\Vtrans(w)\|\le Ml(w)$ and
$\|\slopevec(w)\|\le M\lambda+1$ when $l(w)$ is sufficiently large.
\end{proof}

We close this section by recording for later use a technical lemma which describes sequences
converging to points of $\partial E$.

\begin{lemma}
\label{le:technicalabelianlimset}
Let $(g_n)\subset G$ be a sequence such that $g_nx_0$ converges to a point of $\partial E$,
say $g_n=\grouppair{w_n}{c_n}$.
Then
\[
	\lim_{n\to\infty}
	\frac
		{ d\bigl(x_0,\grouppair{1}{c_n}x_0\bigr) }
		{ d\bigl(x_0,\grouppair{w_n}{1}x_0\bigr) }
			=
	\infty.
\]
\end{lemma}

\begin{proof}
The triangle inequality gives
\[
	\frac{\|\Vtrans(c_n)\|}{\Hdisp(w_n)} \ge \|\slopevec(g_n)\|-\|\slopevec(w_n)\|
\]
The previous two results tell us that $\|\slopevec(g_n)\|\to\infty$ and that $\|\slopevec(w_n)\|$ is bounded.
Therefore
\[
	\frac
		{ d\bigl(x_0,\grouppair{1}{c_n}x_0\bigr)^2 }
		{ d\bigl(x_0,\grouppair{w_n}{1}x_0\bigr)^2 }
=
	\frac
		{\|\Vtrans(c_n)\|^2}
		{\Hdisp(w_n)^2}
		\cdot
	\frac
		{ 1 }
		{ \|\slopevec(w_n)\|^2 + 1 }
			\to
	\infty.
\]
\end{proof}

\section{Cell-Like Limset Maps}
Let \(G=\F_m\times\Z^d\) where $\F_m$ is the free group on $m$ generators where $m\ge 2$.
Assume $G$ acts geometrically on a proper CAT(0) space $X$.
The standard generating set for $G$ is the union of the standard basis for $\Z^d$ along with
$m$ generators for $\F_m$.  The corresponding length metric on $G$
satisfies $l_G(\grouppair{w}{c})=l_{\F_m}(w)+l_{\Z^d}(c)$.
Our proof of Theorem \ref{Thm:CE} here relies heavily on the fact that $\F_m$ is negatively curved.

\subsection{Straight Elements}
Denote $l=l_{\F_m}$.
An element $w\in\F_m$ is called \textit{straight} if it satisfies $l(w^2)=2l(w)$.
Let $\Gamma$ denote the Cayley graph of $\F_m$ with respect to the standard generating set,
an infinite $2m$-valent tree.  Recall that a subspace $A$ of a metric space $B$
is called \textit{quasi-dense} or \textit{$C$-dense} if there is a $C\ge 0$ such that
$B$ is contained in the $C$-neighborhood of $A$.

\begin{lemma}
The following are equivalent for $w\in\F_m$:
\begin{enumerate}
\item $w$ is straight.\\
\item If $a_1...a_n$ is the unique reduced spelling for $w$, then $a_n\neq a_1^{-1}$.\\
\item In the action of $\F_m$ on $\Gamma$, $w$ has an axis passing through the identity.\\
\end{enumerate}
The set of straight elements is 1-dense in $\F_m$ and powers of straight elements are straight.
\end{lemma}

\begin{proof}
Choose $w\in\F_m$ and let $a_1...a_n$ be the reduced
spelling for $w$.  Saying that $w$ is straight means that the word $a_1...a_na_1...a_n$ has no
reductions.  Since $\F_m$ is free, this means that $a_n\neq a_1^{-1}$.  Thus (1) and (2) are equivalent.
To see that (1) and (3) are equivalent simply note that $\Gamma$ is CAT(0) and the identity is in the
minset of $w$ iff $w$ is straight.

Now we show that the set of straight elements is 1-dense.  Suppose $w$ is not straight.  Let
$w=a_1...a_n$ be a spelling in terms of the generating set.
Since we assumed $m>1$, there is a letter $\epsilon$ which is neither $a_1$ nor its inverse.
Since $a_1=a_n^{-1}$, $w\epsilon$ is straight.
\end{proof}

Elements $\grouppair{w}{c}\in G$ will be called \textit{straight} if $w$ is straight.
The reason for considering straight elements is the following.

\begin{lemma}
\label{lemma:axesthrucompactsets}
There is a compact set $D\subset X$ such that every straight element of $G$ has an
axis in $X$ which passes through $D$.
\end{lemma}

Recall that quasi-geodesics in negatively curved spaces behave well \cite[Theorem III.H.1.7]{bridson-haefliger}:

\begin{Theorem}[Stability of Quasi-Geodesics]
For all $\delta>0$, $\lambda\ge 1$, $\epsilon\ge 0$ there exists a constant
$R=R(\delta,\lambda,\epsilon)$ with the following property:

If $X$ is a $\delta$-hyperbolic space, $c$ is a $(\lambda,\epsilon)$-quasi-geodesic
in $X$ and $[p,q]$ is a geodesic segment joining the endpoints of $c$, then the
Hausdorff distance between $[p,q]$ and the image of $c$ is less than $R$.
\end{Theorem}

\begin{proof}[Proof of Lemma \ref{lemma:axesthrucompactsets}]
Let $f:\Gamma\to Y$ denote the QIE.
By the Stability of Quasi-Geodesics Theorem,
there is an $R\ge 0$ such that for every geodesic $\gamma$ in $\Gamma$, $f(\gamma)$
tracks within a Hausdorff distance of $R$ from the geodesic in $Y$ between its endpoints.
In particular, if $\gamma$ begins at a vertex $w$ and ends at a vertex $w'$,
then $f(\gamma)$ is within a Hausdorff distance of $R$ from
$[y_0,wy_0]$.  Now if $w$ is straight, then for every $n$ the geodesic in $\Gamma$ from $w^{-n}$
to $w^n$ passes through the identity.  It follows that the geodesic $[w^{-n}y_0,w^ny_0]$
passes within a distance of $R$ from $y_0$.  Letting $n\to\infty$, we see that every axis for
$w$ passes through $K=\overline{B_R(y_0)}$.  By Proposition \ref{prop:minsetproduct},
if $z\in K$ is in the minset of $w$, then $(z,0)$ is in the minset of $\grouppair{w}{c}$
for any $c\in\Z^d$.  Therefore every straight element of $G$ has an axis passing
through $K\times\{0\}$.
\end{proof}

Recall that two sequences of real numbers $(x_n)$ and $(y_n)$ are called \textit{asymptotic} if
their ratio $x_n/y_n$ converges to 1 as $n\to\infty$ and is written $x_n\sim y_n$.

\begin{lemma}
\label{lemma:reasonforwebbings}
Let $(g_n)\subset G$ be a sequence of straight elements such that $g_nx_0\to\zeta\in\partial X$
and $(k_n)$ be any sequence of positive integers.
Then \(g_n^{k_n}x_0\to\zeta\) as well, and \(d(x_0,g_n^{k_n}x_0)\sim k_nd(x_0,g_nx_0)\).
\end{lemma}

\begin{proof}
Let $\alpha$ be the geodesic ray based at $x_0$ going out to $\zeta$ and
$R>0$ be large enough so that every straight element of $G$ has an axis passing through $B_R(x_0)$.
For each $n$, let \(x_n\in B_n(x_0,R)\) lie in the axis of $g_n$.
Then the geodesics $[x_0,g_n^{k_n}x_0]$ and $[x_0,g_nx_0]$ both stay inside the $R$-tubular neighborhood of
the axis of $g_n$ passing through $x_n$.  It follows that $g_nx_0$ lies inside in the $2R$-neighborhood of
$[x_0,g_n^{k_n}x_0]$.  Thus if $[x_0,g_nx_0]$ stays inside the $1$-neighborhood of $\alpha$ up to time
$T$, then $[x_0,g_n^{k_n}x_0]$ stays inside the $(2R+1)$-neighborhood of $\alpha$ up to time $T$.
It follows that $g_n^{k_n}\to\zeta$.

Now denote $a_n=d(x_0,g_nx_0)$,
$a_n'=d(x_0,g_n^{k_n}x_0)$, $b_n=d(x_n,g_nx_n)$, and $b_n'=d(x_n,g_n^{k_n}x_n)$.  Then
\(b_n\le a_n\le b_n+2R\), \(b_n'\le a_n'\le b_n'+2R\), $b_n'=k_nb_n$, and hence
\[
	\frac{a_n'}{a_n}
		\sim
	\frac{b_n'}{b_n}
		=
	k_n.
\]
Therefore $a_n'\sim k_na_n$ as desired.
\end{proof}

\subsection{Averaging Sequences}
Let \(G=\F_m\times\Z^d\) act geometrically on CAT(0) spaces $X_1$ and $X_2$.
If $m=1$, then $G$ is free abelian and Theorem $\ref{Thm:CE}$ follows from Corollary \ref{co:FTT}.
So we will assume $m>1$.
As before we may assume each space splits as
\(X_i=Y_i\times E_i\) where \(E_i\) is an isometric copy of $\E^d$.
Let $\Vtrans_i$ and $\slopevec_i$ denote the vertical translation and slope vector functions for each $X_i$.
We will denote horizontal displacement in $Y_i$ using $\Hdisp_i$.
Let \(y_i\) be chosen basepoints in $Y_i$, \(x_i=(y_i,0)\), and \( x_0=(x_1,x_2)\).
Given $g\in G$, we will let $\slopescal(g)$ denote the slope of the line segment
$[x_0,g^\Delta x_0]$ in terms of the splitting $ X=X_1\times X_2$.

A fourth piece of information we will need regards slopes in the product $Y_1\times Y_2$.
Given $g\in G$ we consider the geodesic $[x_0,g^\Delta x_0]$ as living in
$X=Y_1\times E_1\times Y_2\times E_2$
and take its coordinate projection $\gamma$ to $Y_1\times Y_2$.
If $\gamma$ is nonconstant, it has a well defined slope which we denote by $\nutoo(g)$.
In other words,
\[
	\nutoo(g)=\frac{\Hdisp_2(g)}{\Hdisp_1(g)}.
\]
This makes sense as long as $g$ does not act only in the $E_1\times E_2$ coordinates.
Since $\F_m$ is torsion free, this is the same as saying that when $g=\grouppair{w}{c}$,
then $w\neq 1$.  Then
\(
	\nutoo(g)=\nutoo(w)
\)
which, by Lemma \ref{lemma:limDGbdd}, is bounded away from 0 and $\infty$.
$\nutoo$ extends continuously to $\partial X-\partial(E_1\times E_2)$.

Denote \(\Lambda=\limset(X,G^\Delta)\) and \(\Sigma=\limset( X,(\Z^d)^\Delta)\)
and let $\phi_i:\Lambda\to\limset(X_i,G)$ denote the limset maps constructed in
Section \ref{sec:schmear}.  The following lemma is needed in light of
Example \ref{ex:limsetproblem}.

\begin{lemma}
\label{lemma:Sigmahomeo}
For each $i$, \(\phi_i^{-1}(\partial E_i)=\Sigma\) and $\phi_i|_{\partial E_i}$ is a homeomorphism.
\end{lemma}

\begin{proof}
By Lemma \ref{lemma:subgroups}, it suffices to prove that \(\phi_i^{-1}(\partial E_i)\subset\Sigma\) for each $i$.
We will prove this for $i=1$.
Choose $\zeta_1\in\partial E_1$, and \(\zeta\in\phi_1^{-1}(\zeta_1)\).  This means that there is
a sequence \((g_n)\subset G\) such that \(g_n^\Delta x_0\to\zeta\) and \(g_nx_1\to\zeta_1\).
Write \(g_n=\grouppair{w_n}{c_n}\).  Let $\lambda\ge 1$ and $\epsilon\ge 0$ be such that
the map $Gx_1\to Gx_0$ is a $(\lambda,\epsilon)$-QIE.
Then
\[
	\frac
		{ d\bigl(x_0,\grouppair{1}{c_n}^\Delta x_0\bigr) }
		{ d\bigl(x_0,\grouppair{w_n}{1}^\Delta x_0\bigr) }
			\ge
	\frac
		{ \frac1\lambda d_1\bigl(x_1,\grouppair{1}{c_n}x_1\bigr) - \epsilon }
		{ \lambda d_1\bigl(x_1,\grouppair{w_n}{1}x_1\bigr) + \epsilon }
	\to\infty
\]
by Lemma \ref{le:technicalabelianlimset}.
Apply convexity of the metric to the geodesics $[x_0,\grouppair{1}{c_n}^\Delta x_0]$
and $[x_0,\grouppair{w_n}{c_n}^\Delta x_0]$ to see that $\grouppair{1}{c_n}^\Delta x_0$ also
converges to $\zeta$.  So $\zeta\in\Sigma$, as desired.
The second fact is Corollary \ref{co:FTT}
\end{proof}

This lemma tells us that preimages of points in $\partial E_i$ are singletons,
and therefore cell-like.  Next we turn our attention to preimages of points
not in $\partial E_i$.
Now,
\[
	\partial X-\partial X_1 \approx (C^o\partial X_1)\times\partial X_2
\]
where $C^o\partial X_1$ denotes the \textit{open cone} on $\partial X_1$:
\(C^o\partial X_1=\partial X_1\ast\{x_0\}-\partial X_1\).
So for each $i$, we have
\begin{align*}
	\partial X_i-\partial E_i
\approx
	\partial Y_i\times C^o\partial E_i \\
\approx
	\partial Y_i\times E_i.
\end{align*}
The first coordinate of this homeomorphism is the extension of the projection
map and the second coordinate is \(\slopevec_i\circ\phi_i\).
Define \(\Lambda^0=\Lambda-\Sigma\), which, by the above reasoning, can be thought of as living in
$\partial Y_1\times E_1\times\partial Y_2\times E_2$.

\begin{prop}
\label{prop:bigcommutingdiagram}
The following diagram commutes:
\[
\begin{array}{ccccc}
	&& \Lambda^0 \\
	& \swarrow && \searrow \\
	\partial Y_1\times E_1 &&&& \partial Y_2\times E_2 \\
	\downarrow              &&&& \downarrow \\
	\partial Y_1            &&&& \partial Y_2 \\
	& \searrow && \swarrow \\
	&& \partial \F_m
\end{array}
\]
\end{prop}

\begin{proof}
Suppose we have a sequence \((\grouppair{w_n}{c_n})\subset G\) such that
\(\grouppair{w_n}{c_n}^\Delta x_0\to\zeta_0\in\Lambda^0\).
Following the maps, we see that for both $i$, the sequences \((w_ny_i)\) converge to points
\(\zeta_i\in\partial Y_i\).
By strong rigidity, the sequence \((w_n)\)
converges to a point \(\zeta_0\in\partial \F_m\).  So for either $i$, $\zeta_0\in\Lambda^0$ gets
mapped to $\zeta_i\in\partial Y_i$ which in turn gets mapped to \(\zeta_0\in\partial \F_m\).
\end{proof}

We now abuse notation by considering $\slopevec_i$ also as functions of $\Lambda^0$; that is, we write
\(\slopevec_i=\slopevec_i\phi_i\).
Given \(\eta\in\partial \F_m\), let \(\Lambda_\eta\) denote its preimage in $\Lambda^0$ according to the diagram
in the previous proposition.

\begin{prop}
\label{prop:reparam}
The map \(\Lambda_\eta\to E_1\times E_2\times(0,\infty)\) given by \((\slopevec_1,\slopevec_2,1/\nutoo)\) is an embedding
and for each $i$, we have a commuting diagram
\[
	\begin{array}{ccc}
		\Lambda_\eta        & \incl & E_1\times E_2\times(0,\infty) \\
			\phi_i\downarrow  &       & 	\downarrow \\
		\{\eta\}\times E_i   & \approx & E_i
	\end{array}
\]
where the map down the right hand side is coordinate projection.
\end{prop}

\begin{proof}
Since $\Lambda$ misses $\partial X_1\cup\partial X_2$, it lives in \(\partial X_1\times\partial X_2\times(0,\infty)\).
The last coordinate here is parameterized by $\slopescal$.
Since the $\phi_i$ map
$\Sigma$ homeomorphically onto $\partial E_i$, \(\phi_i(\Lambda^0)\) lives in \(\partial \F_m\times E_i\) where $E_i$ is parameterized
by $\slopevec_i$.  This means that $\Lambda^0$ lives in
\[
	\partial \F_m\times E_1\times\partial \F_m\times E_2\times(0,\infty)
\]
Now by the previous proposition, the projection to each $\partial\F_m$ is just $\{\eta\}$.
This means that the map
\[
	\Lambda_\eta\to E_1\times E_2\times(0,\infty)
\]
given by \((\slopevec_1,\slopevec_2,\slopescal)\) is an embedding.

To change the third coordinate map to $1/\nutoo$, we just need to know that $\slopescal$ can be written
as a continuous function of $\slopevec_1$, $\slopevec_2$, and $\nutoo$.
Assume \((g_n)\subset G\) is a sequence such that \(g_n^\Delta x_0\to\zeta\in\Lambda^0\).
Then \(\Hdisp_i(g_n)>0\) for large $n$ and we can write
\[
	\slopescal^2(g_n)
\; = \;
	\frac
	{
		d_2(x_2,g_nx_2)^2
	}
	{
		d_1(x_1,g_nx_1)^2
	}
\;=\;
	\nutoo(g_n)^2
	\frac
		{\|\slopevec_2(g_n)\|^2+1}
		{\|\slopevec_1(g_n)\|^2+1}
\]
Letting $n\to\infty$, we get that
\[
	\slopescal(\zeta)
		=
	\nutoo(\zeta)
		\sqrt
		{
			\frac
				{\|\slopevec_2(\zeta)\|^2+1}
				{\|\slopevec_1(\zeta)\|^2+1}
		}
\]
\end{proof}

\setcounter{equation}{-1}

\begin{lemma}[Sequence Averaging]
Choose any $\zeta,\zeta'\in\Lambda_\eta$ with the property that $\slopevec_1(\zeta)=\slopevec_1(\zeta')$.
Then there are sequences $(a_n),(b_n),(c_n)\subset G$ such that $a_n^\Delta x_0\to\zeta$, $b_n^\Delta x_0\to\zeta'$
and all of the following hold:

\begin{minipage}[b]{0.45\linewidth}
\centering
\begin{align}
	\limset( X,&\{c_n\})\;
		\subset\Lambda_\eta
			\label{eq:zero} \\
	\Hdisp_1(a_n)+\Hdisp_1(b_n)
		&\sim
	\Hdisp_1(c_n)
			\label{eq:one} \\
	\Hdisp_2(a_n)+\Hdisp_2(b_n)
		&\sim
	\Hdisp_2(c_n)
		\label{eq:two} \\
	\Hdisp_2(a_n)
		&\sim
	\Hdisp_2(b_n)
		\label{eq:three}
\end{align}
\end{minipage}
\hspace{0.5cm}
\begin{minipage}[b]{0.45\linewidth}
\centering
\begin{align}
		\slopevec_1(c_n)
	&\to
		\slopevec_1(\zeta)
			\label{eq:four} \\
		\slopevec_2(c_n)
	&\to
		\frac
		{
			\slopevec_2(\zeta)+\slopevec_2(\zeta')
		}{2}
			\label{eq:five} \\
		\nutoo(c_n)^{-1}
	&\to
		\frac
		{
			\nutoo(\zeta)^{-1}+\nutoo(\zeta')^{-1}
		}{2}
			\label{eq:six}
\end{align}
\end{minipage}
\end{lemma}

\begin{proof}
Since the collection of straight elements is quasi-dense in $G$ and $G^\Delta\to X$ is a QIE, we can get
sequences of straight elements \(g_n=\grouppair{v_n}{\rho_n}\) and \(h_n=\grouppair{w_n}{\sigma_n}\) such that
$g_n^\Delta x_0\to\zeta$ and $h_n^\Delta x_0\to\zeta'$.  Since $\zeta,\zeta'\notin\Sigma$,
$\Hdisp_2(w_n)$ and $\Hdisp_2(v_n)$ both go to infinity.
Define $s_n$ and $t_n$ to be the floors of
$\Hdisp_2(w_n)$ and $\Hdisp_2(v_n)$.

We will take \(a_n=g_n^{s_n}\), \(b_n=h_n^{t_n}\), and
$c_n=a_nb_n=\grouppair{v_n^{s_n}w_n^{t_n}}{\rho_n^{s_n}\sigma_n^{t_n}}$.
The first two converge to
$\zeta$ and $\zeta'$ by Lemma \ref{lemma:reasonforwebbings}.
Now, since $v_n$ and $w_n$ both converge to $\eta\in\partial\F_m$,
reduced spellings for $v_n$ and $w_n$ share the same first letter $\epsilon$ when $n$ is large.
Since they are straight, they do not end in $\epsilon^{-1}$.  Using the fact that $\F_m$ is free,
a geodesic edge path $\gamma_n$ from $1$ to $v_n^{s_n}w_n^{t_n}$
in $\Gamma$ passes through $v_n$.  Since $v_n\to\eta$, so does $v_n^{s_n}w_n^{t_n}$.
This gives (\ref{eq:zero}).

By the same reasoning, $\gamma_n$ also passes through the vertex $v_n^{s_n}$.
Since the $Y_i$ are $\delta$-hyperbolic, it follows from the Stability of Quasi-Geodesics
Theorem that there is an $R$ such that for both $i$, $v_n^{s_n}y_i$ is inside the $R$-neighborhood of
$[y_i,v_n^{s_n}w_n^{t_n}y_i]$.  By convexity of the CAT(0) metric, we get
(\ref{eq:one}) and (\ref{eq:two}).

For (\ref{eq:three}), we apply Lemma \ref{lemma:reasonforwebbings}:
\[
	\frac
		{\Hdisp_2(v_n^{s_n})}
		{\Hdisp_2(w_n^{t_n})}
\sim
	\frac
		{s_n\Hdisp_2(v_n)}
		{t_n\Hdisp_2(w_n)}
\to
	1.
\]
For (\ref{eq:four}) and (\ref{eq:five}) we will use the following equation which is easy to check:
\[
(\dagger)\qquad
	\slopevec_i(c_n)
=
	\slopevec_i(a_n)\frac{\Hdisp_i(a_n)}{\Hdisp_i(c_n)}
		+
	\slopevec_i(b_n)\frac{\Hdisp_i(b_n)}{\Hdisp_i(c_n)}.
\]
If $i=1$, we use $(\dagger)$ to compute
\begin{align*}
	\bigl|
		\slopevec_1(c_n) - \slopevec_1(b_n)
	\bigr|
&=
	\left|
		\slopevec_1(a_n)
			\smallfrac{\Hdisp_1(a_n)}{\Hdisp_1(c_n)}
				+
		\slopevec_1(b_n)
		\left(
			\smallfrac{\Hdisp_1(b_n)}{\Hdisp_1(c_n)}
				-
			1
		\right)
	\right| \\
&\le
	\left|
		\slopevec_1(a_n)
			\smallfrac{\Hdisp_1(a_n)}{\Hdisp_1(c_n)}
				-
		\slopevec_1(b_n)
			\smallfrac{\Hdisp_1(a_n)}{\Hdisp_1(c_n)}
	\right|
	+
	\left|
		\slopevec_1(b_n)
			\smallfrac{\Hdisp_1(a_n)}{\Hdisp_1(c_n)}
			+
		\slopevec_1(b_n)
		\left(
			\smallfrac{\Hdisp_1(b_n)}{\Hdisp_1(c_n)}
				-
			1
		\right)
	\right| \\
&=
	\smallfrac{\Hdisp_1(a_n)}{\Hdisp_1(c_n)}
	\bigl|
		\slopevec_1(a_n)-\slopevec_1(b_n)
	\bigr|
		+
	\slopevec_1(b_n)
	\left|
		\smallfrac{\Hdisp_1(a_n)+\Hdisp_1(b_n)}{\Hdisp_1(c_n)}
			-
		1
	\right|.
\end{align*}
Using (1) and the facts that $0\le\Hdisp_1(a_n)\le\Hdisp_1(c_n)$
and \(\slopevec_1(a_n)\) and \(\slopevec_1(b_n)\) both converge to $\slopevec_1(\zeta)<\infty$, this all goes to zero as $n\to\infty$
giving (\ref{eq:four}).

Using (\ref{eq:two}) and (\ref{eq:three}) we get
\[
	\frac{\Hdisp_2(a_n)}{\Hdisp_2(c_n)}
		\sim
	\frac{\Hdisp_2(b_n)}{\Hdisp_2(c_n)}
		\to
	\frac12.
\]
Apply $(\dagger)$ and let $n\to\infty$ to get (\ref{eq:five}).

Finally, we use (\ref{eq:one}), (\ref{eq:two}), and (\ref{eq:three}) to compute (\ref{eq:six}):
\begin{align*}
	\nutoo(c_n)^{-1}
&\sim
	\frac
	{
		\frac
			{\Hdisp_2(a_n)}
			{\Hdisp_2(b_n)}
			\nutoo(a_n)^{-1}
				+
			\nutoo(b_n)^{-1}
	}
	{
		\frac
			{\Hdisp_2(a_n)}
			{\Hdisp_2(b_n)}
				+
			1
	} \\
&\to
	\frac{\nutoo(\zeta)^{-1}+\nutoo(\zeta')^{-1}}{2}.
\end{align*}
\end{proof}

\begin{proof}[Proof of Theorem \ref{Thm:CE}]
By Lemma \ref{lemma:Sigmahomeo} we know that preimages of points of $\partial E_1$ are just points.
It remains to consider preimages of points of $\partial X_1-\partial E_1$.
Choose $\eta\in\partial \F_m$.
Proposition \ref{prop:reparam} provided an
embedding \(\Lambda_\eta\incl E_1\times E_2\times(0,\infty)\subset\E^{2d+1}\) by the map \((\slopevec_1,\slopevec_2,\nutoo^{-1})\).
Choose any \(\zeta_1\in\phi_1(\Lambda_\eta)\)
and \(\zeta,\zeta'\in\phi_1^{-1}(\zeta_1)\).  This means that \(\slopevec_1(\zeta)=\slopevec_1(\zeta')=\slopevec_1(\zeta_1)\).

Get sequences $(a_n),(b_n),(c_n)\subset G$ such that $a_n^\Delta x_0\to\zeta$ and $b_n^\Delta x_0\to\zeta'$,
as prescribed by the previous lemma.
From (\ref{eq:zero}), (\ref{eq:four}), (\ref{eq:five}), and (\ref{eq:six}) it follows
that $(c_n)$ converges to a point $\zeta''\in\Lambda_\eta$.   (\ref{eq:four}) tells us that
$\zeta''\in\phi_1^{-1}(\zeta_1)$ and (\ref{eq:five}) and (\ref{eq:six}) tell us that $\zeta''$
is actually the midpoint of the line segment $[\zeta,\zeta']$.
By applying this averaging
process repeatedly, we see that a dense subset of $[\zeta,\zeta']$ is contained in $\phi_1^{-1}(\zeta_1)$.
But of course $\phi_1^{-1}(\zeta_1)$ is closed, which means that $[\zeta,\zeta']\subset\phi_1^{-1}(\zeta_1)$.
This proves that $\phi_1^{-1}(\zeta_1)$ is a convex subspace of $\E^{2d+1}$.
Convex subspaces of euclidean spaces are disks.
The proof for $\phi_2$ is the same.
\end{proof}

\begin{proof}[Proof of Theorem \ref{Thm:NotComparable}]
We need to prove that the only limset maps between boundaries of $\F_2\times\Z$ are homeomorphisms.
Assume $G=\F_2\times\Z$ acts geometrically on two CAT(0) spaces $X_1$ and $X_2$ with specified basepoints
$x_1$ and $x_2$
and that we have a limset map \(\rho:\partial X_1\to\partial X_2\) which is not a homeomorphism.
If $\rho$ were injective, then by Lemma \ref{le:characterization of existence of limset maps},
$\rho$ would be a homeomorphism, giving us a contradiction.

So $\rho$ must not be injective.
It follows from Lemma \ref{lemma:Sigmahomeo} that $\rho$ takes the poles of $\partial X_1$
to the poles of $\partial X_2$, and from Proposition \ref{prop:bigcommutingdiagram} that
$\rho$ takes longitudes to longitudes.  So there are distinct $\zeta_1,\zeta_1'\in\partial X_1$ lying
in the same longitude such that $\rho(\zeta_1)=\rho(\zeta_1')$; call this point $\zeta_2$.
Let $c$ be the generator of $\Z$ and assume our coordinate system has been taken so that
\(\Vtrans_1(c)\) and \(\Vtrans_2(c)\) are both positive.
Choose sequences \((g_n),(g_n')\subset G\) such that $(g_nx_1)$ and $(g_n'x_1)$ converge to $\zeta_1$ and $\zeta_1'$,
say \(g_n=\grouppair{w_n}{c^{k_n}}\) and \(g_n'=\grouppair{w_n'}{c^{k_n'}}\).

Now the slopes and vertical displacements here are 1-dimensional vectors and, on these sequences,
real-valued.  To emphasize this, we will drop the vector notation and call them simply
${\sf m}_1$, ${\sf m}_2$, ${\sf t}_1$, and ${\sf t}_2$.
If ${\sf m}_1(g_n')\sim {\sf m}_1(w_n)$, then set $g_n''=\grouppair{w_n}{1}$.
Otherwise choose positive integers $i_n$ such that
\[
	i_n
		\sim
	\bigl[
		{\sf m}_1(g_n')
			-
		{\sf m}_1(w_n)
	\bigr]
	\frac
		{ \Hdisp_1(w_n) }
		{ {\sf t}_1(c) }.
\]
and take $g_n''=\grouppair{w_n}{c^{i_n}}$.  Either way, ${\sf m}_1(g_n'')\sim{\sf m}_1(g_n')$,
which means that $g_n''x_1\to\zeta_1'$.  By hypothesis, $\partial X_2<\partial X_1$, so
$g_n''x_2\to\zeta_2$, and $\lim_{n\to\infty}{\sf m}_2(g_n'')=\lim_{n\to\infty}{\sf m}_2(g_n)$.
So
\[
	{\sf m}_1(g_n'')-{\sf m}_1(g_n)
		\sim
	\bigl[
		{\sf m}_2(g_n'')-{\sf m}_2(g_n)
	\bigr]
		\frac
			{ {\sf t}_1(c) }
			{ {\sf t}_2(c) }
		\frac
			{ \Hdisp_2(w_n) }
			{ \Hdisp_1(w_n) }
				\to 0
\]
because Lemma \ref{lemma:limDGbdd} guarantees the ratio of the
$\Hdisp_i$'s remains bounded.  But this of course means that ${\sf m}_1(g_n)\to\zeta_1'$
contradicting the fact that $\zeta_1'\neq\zeta_1$.
\end{proof}

\bibliography{bean}{}
\bibliographystyle{siam}

\end{document}